%% file: main.tex
\documentclass[onefignum,onetabnum]{siamonline250106}

\graphicspath{{images/}}

\input{ex_shared}

\ifpdf
\hypersetup{
  pdftitle={An Example Article},
  pdfauthor={}
}
\fi

\usepackage{tikz}
\usetikzlibrary{spy} 
\pgfplotsset{compat=1.18}

\begin{document}

\maketitle

\begin{abstract}
    We propose a method for sampling from Gibbs distributions of the form $\pi(x)\propto\exp(-U(x))$ that leverages a family $(\pi^{t})_t$ of approximations of the target density which is deliberately constructed such that $\pi^{t}$ exhibits favorable properties for sampling when $t$ is large, and such that $\pi^{t}$ approaches \( \pi \) as $t$ approaches 0. This sequence is obtained by replacing (parts of) the potential $U$ with its Moreau envelope. Through the sequential sampling from $\pi^{t}$ for decreasing values of $t$ by a Langevin algorithm with appropriate step size, the samples are guided from a simple starting density to the more complex target quickly. We prove that \( t \mapsto \pi^t  \) is Lipschitz continuous in the total variation distance and Hölder continuous in the Wasserstein-$p$ distance, that the sampling algorithm is ergodic, and that it converges to the target density without assuming convexity or differentiability of the potential $U$. In addition to the theoretical analysis, we show experimental results that support the superiority of the method in terms of convergence speed and mode-coverage of multi-modal densities over current algorithms. The experiments range from one-dimensional toy-problems to high-dimensional inverse imaging problems with learned potentials.
\end{abstract}

\begin{keywords}
Langevin diffusion, Markov chain Monte Carlo, sampling, inverse imaging
\end{keywords}

\begin{MSCcodes}
65C40, 65C05, 68U10, 65C60
\end{MSCcodes}

\section{Introduction}
This article is concerned with sampling from distributions of the form
\begin{equation}\label{eq:potential}
    \pi(x) = \frac{\exp{\bigl(-U(x) \bigr)}}{\int \exp{\bigl( -U(y) \bigr)}\, \dd y}.
\end{equation}
The potential $U = F + G$ is composed of the functions $F, G:\R^d\rightarrow [0,\infty)$ and is such that $\int \exp{\bigl( -U(x) \bigr)}\, \dd x<\infty$.
The consideration of the additive structure $U=F+G$ is motivated by frequent applications where $U$ can be split into a well-behaved function $F$ and a more difficult-to-handle function $G$ (\cf \cite{narnhofer2022posterior,kobler2020total}).
In particular, we cover settings where $G$ and, consequently, $U$ are nonconvex or nondifferentiable.

Sampling from distributions of the form \eqref{eq:potential} is a task that frequently arises in, \eg, Bayesian inverse problems\cite{luo2023bayesian}, mathematical imaging\cite{pereyra2016proximal,durmus2018efficient}, machine learning\cite{zach2023stable,zach2021computed}, and uncertainty quantification\cite{narnhofer2022posterior}.
As an example, training strategies for machine learning such as maximum-likelihood estimation \cite{du2020improved,zach2023stable,nijkamp2020anatomy}, which are becoming increasingly popular, often rely on sampling algorithms as subroutines. 
For most interesting potentials, sampling from \eqref{eq:potential} necessitates \gls{mcmc} methods where a Markov chain $(X_k)_{k\geq1}\subset \R^d$ is deliberately designed such that the law of $X_k$ converges to the target distribution $\pi$ as $k\rightarrow\infty$.
To tackle high-dimensional problems often encountered in imaging applications, Markov chains derived from discretizations of the Langevin diffusion \gls{sde}
\begin{equation}\label{eq:langevin_sde}
    \dd X_t = -\nabla U (X_t)\, \dd t + \sqrt{2}\,\dd W_t,
\end{equation}
where $(W_t)_t$ denotes Brownian motion, have become popular. Note, however, that \eqref{eq:langevin_sde} requires $U$ to be differentiable.
A straightforward way to obtain a Markov chain that approximates \eqref{eq:langevin_sde} is via the \gls{em} discretization that leads to the \gls{ula}
\begin{equation}\label{eq:euler_maruyama}
    X_{k+1} = X_k -\tau\nabla U(X_k) + \sqrt{2\tau} Z_k,
\end{equation}
where $\tau>0$ is the step size of the discretization and $(Z_k)_k$ are \gls{iid} standard Gaussian random vectors.
Various convergence results of such schemes (under stronger assumptions than in this paper) can be found in \cite{roberts1996exponential,durmus2019analysis,durmus2019high,durmus2017nonasymptotic,dalalyan2016theoretical,lamberton2003recursive,habring2024subgradient,fruehwirth2024ergodicity}.
While more sophisticated schemes have been proposed, the \gls{em} discretization remains the most popular due to the simple implementation and diminishing returns for higher-order schemes in practice.

Practitioners face two challenges when using \gls{ula}: Slow mixing and strong assumptions on the potential.
The former ultimately led to the birth of diffusion models in~\cite{song2019generative}, while the latter sparked research on Langevin algorithms that are applicable to nondifferentiable and nonconvex potentials~\cite{durmus2018efficient,luu2021sampling}.
In this work, we propose \gls{daz} as a novel sampling method that combines ideas from diffusion models with ideas from nondifferentiable Langevin sampling.
In particular, we propose the successive Langevin sampling of a sequence of distributions $\pi^t$ for \( t \to 0 \) which is obtained by replacing $G$ with its \MoreauName{} envelope with parameter $t$.
The \MoreauName{} envelope convexifies the potential (we formalize this statement in \Cref{lemma:non-convexity}), allows using larger step sizes, and, consequently, has favorable mode-coverage behavior when $t$ is large and converges to the target as $t\rightarrow 0$.
In addition and in contrast to diffusion models, the proposed sampling algorithm does not require a trained model for all \( t > 0 \) since the \gls{mmse} denoising step is replaced with a \gls{map} denoising problem that can be solved with efficient optimization algorithms for a large class of potentials.
Consequently, the proposed method is applicable to a large class of models used for inverse problems.
\paragraph{Contributions}
We provide the following contributions with regards to the proposed sampling algorithm that is summarized in \Cref{algo}:
\begin{enumerate}
    \item In the nondifferentiable and nonconvex setting, we prove the exponential ergodicity of the Langevin diffusion as well as the geometric ergodicity of the corresponding discretization for any fixed Moreau parameter \( t \).
    Moreover, we show that the stationary distribution of the discretization converges to that of the continuous-time Langevin diffusion as the step size vanishes (\Cref{prop:ergodic1,prop:erg_small_t}).
    \item We show that the map $t\mapsto \pi^t$ is Lipschitz continuous with respect to the total variation distance, as well as Hölder continuous with Hölder exponent $p\in [1,\infty)$ under appropriate integrability assumptions on the density $\pi^t$ (\Cref{prop:curve_lipschitz}). This continuity ensures that a small change in $t$ implies a small change in \( \pi^t \), which is a crucial requirement for a meaningful annealing strategy.
    \item We relate the proposed sampling algorithm to diffusion models by showing that the distributions $(\pi^t)_t$ can be understood as a \emph{zero-temperature} limit of the corresponding diffusion distributions (\Cref{prop:diffusion_potential,prop:diffusion_gradient}).
    \item We provide an extensive set of numerical experiments.
        We demonstrate the efficiency of the proposed sampling algorithm quantitatively on toy examples that allow for an efficient computation of reference distributions and qualitatively on several high-dimensional applications in imaging that leverage potentials that are learned from data.
\end{enumerate}

\begin{algorithm}[t]
\caption{Diffusion at Absolute Zero (DAZ)}\label{algo}
\begin{algorithmic}[1]
\refstepcounter{alglinenumber}
\Require Number of Moreau levels $N$ with schedule $\{t_n\}_{n=1}^N$ and corresponding Langevin step sizes $\{\tau_n\}_{n=1}^N$, number of Langevin steps per Moreau level $K$, initial condition \( X^N_1 \) \refstepcounter{alglinenumber}
\For{$n=N,\dots, 1$}\refstepcounter{alglinenumber}
    \For{$k=1,\dots K$}
        \State $Z_k^n\sim\mathcal{N}(0,I)$
        \State $X_{k+1}^n = X_k^{n} - \tau_n \nabla F(X_k^n)
                - \frac{\tau_n}{t_n}\left(X_k^n - {\prox}_{t_n G}(X_k^n) \right) + \sqrt{2\tau_n} Z_k^n$ \refstepcounter{alglinenumber}\label{alg:update}
    \EndFor
    \State $X_{0}^{n-1} = X_{K}^n$
\EndFor
\State \Return $X_{K}^{1}$
\end{algorithmic}
\end{algorithm}

\section{Related Work}\label{sec:related}
In recent years, data-driven approaches have significantly influenced applied mathematics and related fields.
A central aspect of these data-driven approaches is their close relation to probabilistic modeling, which naturally led to an increased interest in sampling strategies.
In the following, we provide an overview of relevant works that are closely related to the present article.
To facilitate the comparison between our work and some of these works, we revisit some of them---after the presentation of our results---in more detail in~\cref{ssec:comparison}.
\paragraph{Langevin sampling}
A large body of work investigates conditions for ergodicity of the Langevin diffusion \eqref{eq:langevin_sde} and \gls{ula} \eqref{eq:euler_maruyama} and the relation of the respective stationary distributions to the target under various assumptions on the potential \( U \).
Under the assumption that \( U \) is differentiable, the authors of the early work \cite{roberts1996exponential} show exponential ergodicity of the diffusion with stationary measure $\pi$, as well as geometric ergodicity of \gls{ula} under growth conditions on $\nabla U$.
Later works focused on establishing nonasymptotic convergence results with explicit rates.
For strongly convex (at least outside a ball) and differentiable potentials with Lipschitz continuous gradient, explicit rates are provided in \cite{dalalyan2016theoretical,cheng2018sharp}.
Weaker conditions on the growth of $U$ or $\nabla U$ are considered in \cite{durmus2017nonasymptotic}.
In \cite{durmus2019analysis}, \gls{ula} (and similar schemes) are interpreted as an iterative minimization of the Kullback-Leibler divergence to the target density (\cf \cite{variational1998jordan}) and the authors provide nonasymptotic convergence results.
 
The popularity of nondifferentiable regularizers (see, \eg, \cite{bredies2010total,rudin1992nonlinear}) in Bayesian inverse problems in imaging sparked research on sampling algorithms that can handle nondifferentiable potentials.
Methods that rely on the subgradient or the proximal map of (nondifferentiable parts of) $U$ are proposed and analyzed in \cite{durmus2019analysis,habring2024subgradient,fruehwirth2024ergodicity}.
The resulting sampling algorithms closely resemble proximal-gradient-style methods from optimization.
In \cite{burger2024coupling,narnhofer2022posterior}, a primal-dual sampling scheme inspired by the Chambolle-Pock method for optimization \cite{chambolle2011first} is used to tackle nondifferentiable potentials.
The primal-dual-inspired methods are ad-hoc algorithms that have been demonstrated to work well in practice (see, \eg, \cite{narnhofer2022posterior}), but convergence has only been proven under strong differentiability assumptions \cite{burger2024coupling}.
An alternative approach for sampling from nondifferentiable densities is to replace the target density with a differentiable surrogate.
This has been proposed, \eg, in \cite{durmus2018efficient,pereyra2016proximal} where the nondifferentiable part of the potential is substituted with its \MoreauName{} envelope;
the authors named the resulting algorithm \gls{myula}.
In \cite{ehrhardt2024proximal}, the convergence of such methods is analyzed when the proximal map is inexact.
 
The \MoreauName{} envelope-based approach from \cite{durmus2018efficient,pereyra2016proximal} is extended to nonconvex potentials in~\cite{luu2021sampling} where a result about convergence of the \gls{em} discretization to the continuous-time diffusion in expectation for finite time is presented \cite[Theorem 1]{luu2021sampling}.
We extend this by an extensive \emph{ergodic analysis} of the Langevin diffusion and its discretization and a proof of the convergence of the stationary distribution of the discrete chain to the target as the step size vanishes (in the \gls{tv} norm, among others).
Moreover, in \cite{luu2021sampling} like in \cite{durmus2018efficient}, the \MoreauName{} envelope was used  to obtain a single surrogate density for $\pi$.
In contrast, we consider a sequence of densities $(\pi^t)_t$ with different \MoreauName{} parameters that approach $\pi$, similar to diffusion models.
The use of different \MoreauName{} parameters necessitates distinguishing the case of small parameters that lead to differentiable envelopes and large parameters where differentiability is not guaranteed.
In both cases, we prove the ergodicity of the discrete chain.

In \cite{pereyra2020skrock} the authors propose to replace the simple \gls{em} discretization used in \gls{ula} by a Runge-Kutta stochastic integration scheme which extends the deterministic Chebyshev method \cite{skvortsov2011explicit} to SDEs. The proposed method is coined \gls{rock}. While theoretical results about the convergence speed of \gls{rock} are unfortunately not yet available \cite[Section 3.1.1]{pereyra2020skrock}, the method provides a significant acceleration in practice (see also \cref{sec:experiments}).

\paragraph{Annealed Langevin sampling}
A particular inspiration for the present article has been annealed Langevin sampling \cite{song2019generative}, which later lead to diffusion models \cite{song2020score}. 
Annealed Langevin sampling is related to \gls{daz} in the sense that both approaches propose to approximate the target $\pi$ by a family of distributions $(\pi^t)_t$ which is designed so that $\pi^t$ exhibits favorable properties for large $t$ and so that $\pi^t\rightarrow \pi$ as $t\rightarrow 0$ in an appropriate sense. The difference between the two approaches lies in the definition of the family $\pi^t$. In annealed Langevin sampling, $\pi^t$ is defined as the convolution of $\pi$ with a Gaussian of variance $t$, \ie, $\pi^{t} = \pi*\Nc(0,\sqrt{t})$. In \gls{daz}, $\pi^t$ is defined by replacing $G$ by its Moreau envelope. In order to draw from $\pi$, both \gls{daz} and annealed Langevin sampling sample successively from their respective regularized distributions $\pi^{t_k}$ for $k=n,n-1,\dotsc, 0$ using \gls{ula} where $t_n>t_{n-1}>\dots>t_0$ is a predefined parameter schedule. Initially, for $t_k$ large, the distribution $\pi^{t_k}$ admits a rather favorable structure and, thus, allows for efficient sampling. As $t_k$ is decreased, the samples are guided to the complex target distribution.

A crucial difference between annealed Langevin sampling and the proposed method is their applicability to a given and generic (\ie, without special structure and not trained in a method-specific manner) potential \( U \).
In particular, when \( \pi^t = \pi*\Nc(0,\sqrt{t}) \) is constructed for annealed Langevin sampling, one evaluation of $\nabla \log \pi^t$ (which is required for the Langevin sampling) at any point requires the computation of the \gls{mmse}-estimate of the denoising of that point under the prior \( \pi \).
For a general potential, for instance a deep neural network, this task is as hard as the original problem of sampling from \( \pi \).
This challenge is typically circumvented by the off-line direct training of a family of such \gls{mmse} denoisers~\cite{song2019generative}.
However, this renders the method impractical for the sampling of a given potential \( U \).
In contrast, when \( \pi^t(x) \propto \exp\bigl( -U^t(x) \bigr) \) is constructed for \gls{daz}, one evaluation of $\nabla \log \pi^t$ can be computed efficiently whenever the proximal map of $G$ can be computed efficiently.
This can be accomplished by the resolution of a \gls{map} denoising problem under a prior with potential \( G \) (see~\cref{def:moreau envelope} and the discussion thereafter) by efficient algorithms for a large class of functions that includes neural networks.
Consequently, the proposed method can be readily applied to a large class of potentials.
The relation between annealed Langevin sampling and \gls{daz} is formalized more rigorously in \Cref{prop:diffusion_potential,prop:zero_temp_TV,prop:diffusion_gradient} where we show that the sequence of distributions used in the present work can be viewed as a zero-temperature limit of the sequence of distributions used in diffusion models.

\paragraph{Successive regularization}
\Gls{daz} can be understood as a sampling analog to a successive regularization approach that has recently been proposed in the context of optimization.
For the minimization of a possibly nonconvex and nondifferentiable function $H$, the authors of \cite{heaton2024global} propose to alternate gradient descent steps on the Moreau envelope of \( H \) and update steps on the Moreau parameter \( t \).
Since the Moreau envelope leaves global minima unchanged and, under certain conditions, local minimizers of the Moreau envelope are global minimizers of $H$ \cite[Lemma A.6]{heaton2024global}, this constitutes a valid approach to finding the global minimizers of \( H \).
However, while for our sampling approach we require that the Moreau parameter approaches \( 0 \), this need not be the case for optimization \cite{heaton2024global}.

\section{Preliminaries}
In this section, we introduce some notations and mathematical preliminaries that are used throughout the paper.
We start with the definition of the (regular) subdifferential, which is a crucial part in the rigorous treatment of nondifferentiable functions.
\begin{definition}[Regular subdifferential]
    The \emph{(regular) subdifferential} of a function $H:\R^d\rightarrow\R$ at the point $x \in \R^d$ is the set
\[
    \partial H(x) = \left\{v\in\R^d\;\middle|\; \liminf_{\substack{y\rightarrow x\\y\neq x}}\frac{H(y) - H(x) -  \langle v, y-x\rangle}{\|y-x\|} \geq 0\right\}.
\]
An element $v\in\partial H(x)$ is referred to as a \emph{(regular) subgradient} of $H$ at $x$. 
\end{definition}
Our strategy relies heavily on the \MoreauName{} envelope, which we now formally define.
\begin{definition}[\MoreauName{} envelope]\label{def:moreau envelope}
    For a function $H:\R^d\rightarrow \R$ the \emph{\MoreauName{} envelope} with \emph{\MoreauName{} parameter} $t>0$ is defined as 
    \begin{equation}\label{eq:moreau}
        M_H^t(x) \coloneqq \inf_{y\in\R^d} H(y)+\tfrac{1}{2t}\|x-y\|^2.
    \end{equation}
\end{definition}
The \MoreauName{} envelope has the property that $M_H^t(x) \leq H(x)$ for all \( x \in \R^d \) and \( t > 0 \) and, if $H$ is proper, \gls{lsc} and there exists $t>0$ such that $M_H^t(x)>-\infty$ for some $x\in \R^d$, the map $(x,t) \mapsto M_H^t(x)$ is continuous on \( \R^d \times (0, \infty) \).
In addition $M^t_H(x) \rightarrow H(x)$ for all \( x \in \R^d \) as $t\rightarrow 0$, which motivates its use it as an approximation of $H$ \cite[Theorem 1.25]{rockafellar2009variational}.
\begin{definition}[Proximal map]
    The \emph{proximal map} of a function \( H : \R^d \to \R \) with \MoreauName{} parameter \( t > 0 \) assigns to any \( x \in \R^d \) the (possibly multi-valued or empty) set
    \[
        \prox_{tH}(x) = \operatorname*{argmin}_{y \in \R^d} H(y) + \tfrac{1}{2t} \| x - y \|^2.
    \]
\end{definition}
A crucial relationship is that $\nabla M_H^t(x) \in \frac{1}{t}\bigl(x-{\prox}_{t H}(x)\bigr)$ at all points $x \in \R^d$ where $M_H^t$ is differentiable \cite[Example 10.32]{rockafellar2009variational}.
This yields a practical way of computing the gradient of $M^t_H$ by solving an optimization problem.
\paragraph{Probability}
Let $\Bc(\R^d)$ be the Borel $\sigma$-algebra on $\R^d$.
We denote the space of all probability measures on $\Bc(\R^d)$ as $\Pc(\R^d)$ and the subspace of all probability measures with bounded $p$-th moment as $\Pc_p(\R^d) \coloneqq\left\{ \mu\in\Pc(\R^d)\;|\; \int\|x\|^p\;\dd\mu(x)<\infty\right\}$.
The Dirac measure concentrated at $x\in \R^d$ is denoted as $\delta_x\in \Pc(\R^d)$.
For two probability measures $\mu,\nu\in \Pc(\R^d)$, we denote their Wasserstein-$p$ distance as
\[
    \mathcal{W}_p(\mu,\nu) = \inf_{X\sim \mu,\; Y\sim \nu}\E[\|X-Y\|^p]^{\frac{1}{p}} = \inf_{\gamma \in \Pi(\mu,\nu)}\left(\int \|x-y\|^p d \gamma(x,y)\right)^{\frac{1}{p}},
\]
where $\Pi(\mu,\nu)$ denotes the set of all probability measures on $\R^d\times\R^d$ with marginals $\mu$ and $\nu$ \cite[Definition 1.6]{villani2009optimal}, and their \gls{tv} distance as
\[
    \| \mu - \nu \|_{\mathrm{TV}} = \sup_{A \in \Bc(\R^d)} |\mu(A) - \nu(A)|.
\]
Moreover, we define a Markov kernel on $\R^d$ as a map $M: \R^d\times \Bc(\R^d)\rightarrow [0,1]$ such that for any $x\in \R^d$, $M(x,\cdot)$ is a probability measure and for any $A\in \Bc(\R^d)$, $M(\cdot,A)$ is measurable. For $\mu\in \Pc(\R^d)$, we define the probability measure $\mu M\in \Pc(\R^d)$ by
\[
    \mu M(A) = \int M(y,A)\,\mathrm{d}\mu(y)
\]
for any \( A\in \Bc(\R^d) \).
We denote the $k$-fold application of a Markov kernel \( M \) to a measure \( \mu \) as $\mu M^k = \mu M\ldots M$ (\( k \) times).
A Markov kernel can also be interpreted as a linear map that maps bounded and measurable functions to bounded and measurable functions via $f\mapsto Mf$, with $Mf(x) = \int f(y)\,\dd M(x,y)$. 

Let $\mu\in \Bc(\R^d)$.
A Markov semi-group is a family $(P_t)_{t\geq 0}$ where $P_t$ is a linear operator that maps bounded and measurable functions to bounded and measurable functions and is such that $P_t\1=\1$ with $\1$ denoting the constant function with value one, $P_tf\geq 0$ if $f\geq 0$, for every $f\in L^2(\R^d,\mu)$ it holds that $P_tf\rightarrow f$ in $L^2$ as $t\rightarrow 0$, for every $1\leq p<\infty$, $P_t$ extends to a bounded (contraction) operator on $L^p(\R^d,\mu)$, and $P_t\circ P_s = P_{t+s}$, $s,t\geq 0$.
The generator $\Ac$ of a Markov semi-group is an operator defined via 
\[
\Ac f = \lim_{t\rightarrow 0}\frac{1}{t}(P_tf - f)
\]
whose domain consists of all functions $f\in L^2(\R^d,\mu)$ for which the limit exists. For details on Markov semi-groups and their generators, see \cite{bakry2014markov}. In most cases the the Markov semi-group is given as a family of Markov kernels \cite[Section 1.2.2]{bakry2014markov}.
\paragraph{Miscellaneous}
We denote the Lipschitz constant of a Lipschitz continuous function $H:\R^d\rightarrow\R$ as $L_H$.
We denote the ball centered at \( m \in \R^d \) with radius \( R > 0 \) as \( \Ball{R}{m} \coloneqq \{ x \in \R^d \mid \| x - m \| \leq R \} \), and its complement in \( \R^d \) as \( \BallC{R}{m} = \R^d \setminus \Ball{R}{m} \).
\section{Diffusion at Absolute Zero}
We propose \emph{\gls{daz}} to sample from $\pi$ as defined in \eqref{eq:potential}:
\Gls{daz} combines ideas from diffusion methods \cite{song2020score} and nondifferentiable sampling \cite{durmus2018efficient} by considering the sequence of perturbed densities $\pi^t(x) \propto \exp{\bigl(-U^t(x)\bigr)}$ where
\[
    U^t(x)\coloneqq F(x) + M_G^t(x)
\]
for $t>0$.
That is, we replace the function $G$ with its \MoreauName{} envelope, whereas $F$ is left as is.
This splitting strategy is advantageous when the proximal map of \( G \) can be computed efficiently and the proximal map of \( F + G \) is difficult to compute.
However, we show in \Cref{rmk:step_size} that the choice $G=U$ and $F \equiv 0$ has favorable properties with respect to the selectable step sizes when the proximal map of the original potential can be computed efficiently.
As a sampling strategy, we apply annealed Langevin sampling to $(\pi^t)_t$.
The proposed sampling algorithm is summarized in~\Cref{algo}.
There, $0<t_0<t_1<\dots<t_N$ denotes a sequence of \MoreauName{} parameters and $(\tau_n)_n$ a corresponding sequence of step sizes used within the inner loop. This inner loop consists simply of several updates of the conventional \gls{ula} with step size $\tau_n$ applied to the potential $U^{t_n}$, \ie, 
\[
    X^{k+1}_n = X_n^k  -\tau_n \nabla U^{t_n}(X_n^k) +\sqrt{2\tau_n}Z_n^k,
\]
starting from some \( X^0_n \) that is informed from the last iterate of the previous Moreau parameter.
The update step in \cref{alg:update} in \cref{algo} is obtained by plugging in $U^t = F + M^t_G$ as well as $\nabla M^t_G(x) = \frac{1}{t}(x-\prox_{tG}(x))$.
In \cite{durmus2018efficient}, the main motivation for using the \MoreauName{} envelope was to handle nondifferentiable points of $G$.
In the present work, we also exploit the \emph{convexifying} behavior of the \MoreauName{} envelope.
This behavior is exemplified in \Cref{fig:moreau_gm} where the Moreau envelope connects the separated local minima of the potential as $t$ increases.
The convexifying property is now formalized.
\begin{figure}
    \centering
    \begin{tikzpicture}
        \begin{axis}[
            width=0.49\textwidth,
            title=$x \mapsto \exp(-M_G^t(x))$,
            enlargelimits=true, 
            axis on top,         
            clip=true            
        ]
            \addplot graphics [
                xmin=-5, xmax=5,
                ymin=0, ymax=1.6
            ] {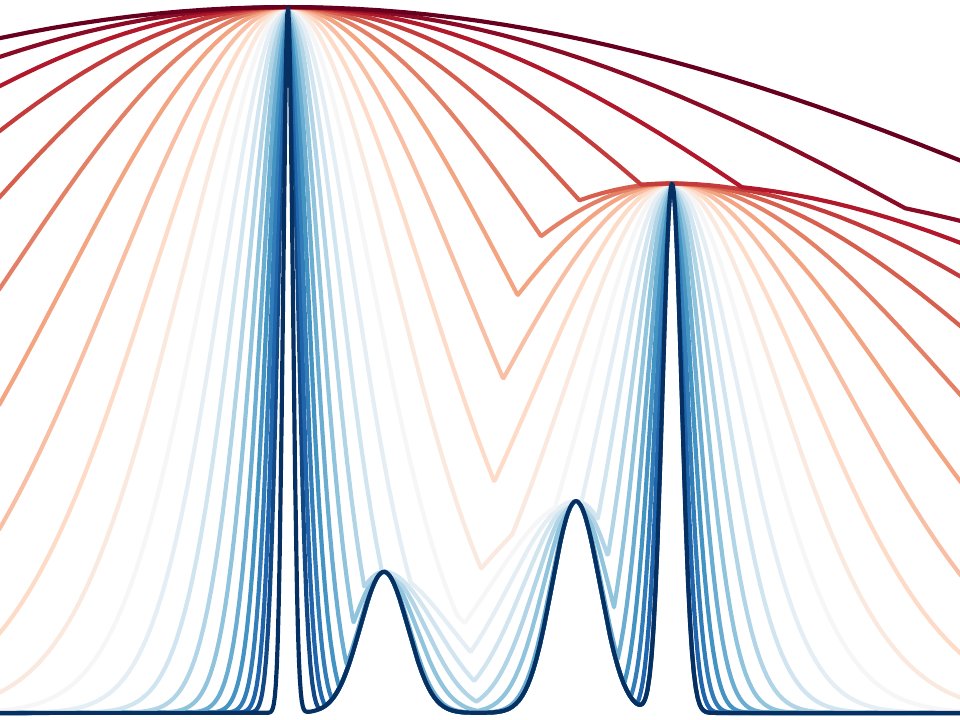};
        \end{axis}
    \end{tikzpicture}
    \hfill
\begin{tikzpicture}[spy using outlines={rectangle, magnification=2, width=4cm, height=0.8cm, connect spies}]
        \begin{axis}[
            width=0.49\textwidth,
            title=$x \mapsto M_G^t(x)$,
            enlargelimits=true, 
            axis on top,         
            clip=true,            
            colormap/RdBu,
            colormap={reverse RdBu}{
            indices of colormap={
                \pgfplotscolormaplastindexof{RdBu},...,0 of RdBu}
            },
            colorbar,
            colorbar style ={
                title=$t$,
                ytick={0.,100},
            },
            point meta min = 0.01,
            point meta max = 100
        ]
            \addplot graphics [
                xmin=-5, xmax=5,
                ymin=0, ymax=129
            ] {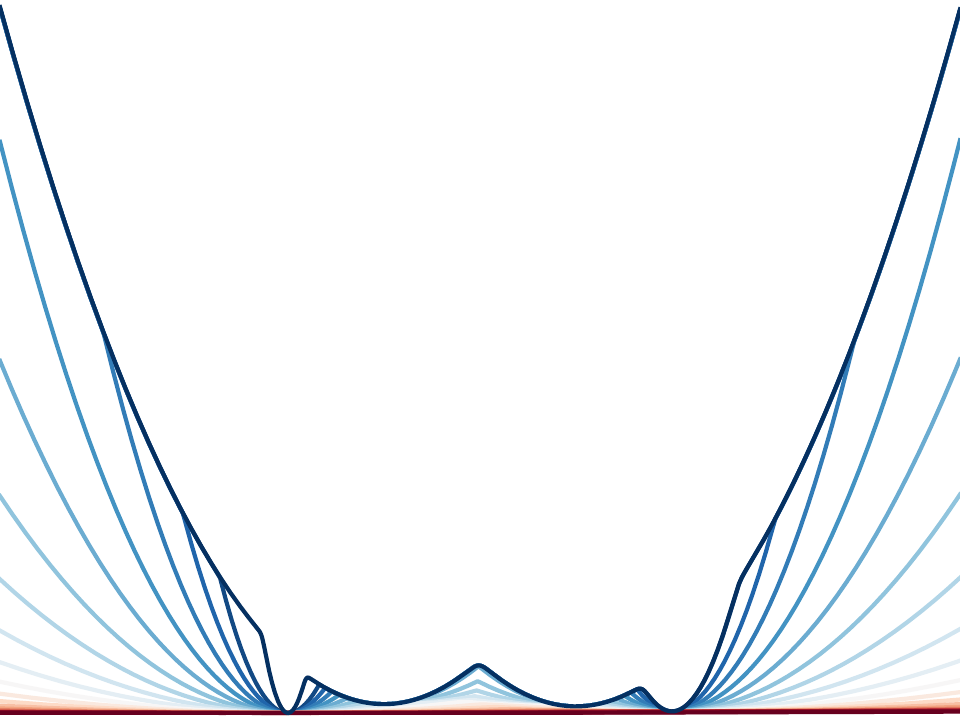};

            \coordinate (spypoint) at (axis cs:0,5.5);
            \coordinate (magnifyglass) at (axis cs:0,100);

        \end{axis}
   \node[rectangle, width=2.01cm, height=0.45cm, draw, black] (glas)  at (spypoint) {};
   \node[rectangle, draw, black, inner sep=1pt,yshift=-0.1cm] (zoom) at (magnifyglass) {\includegraphics[width=3.2cm]{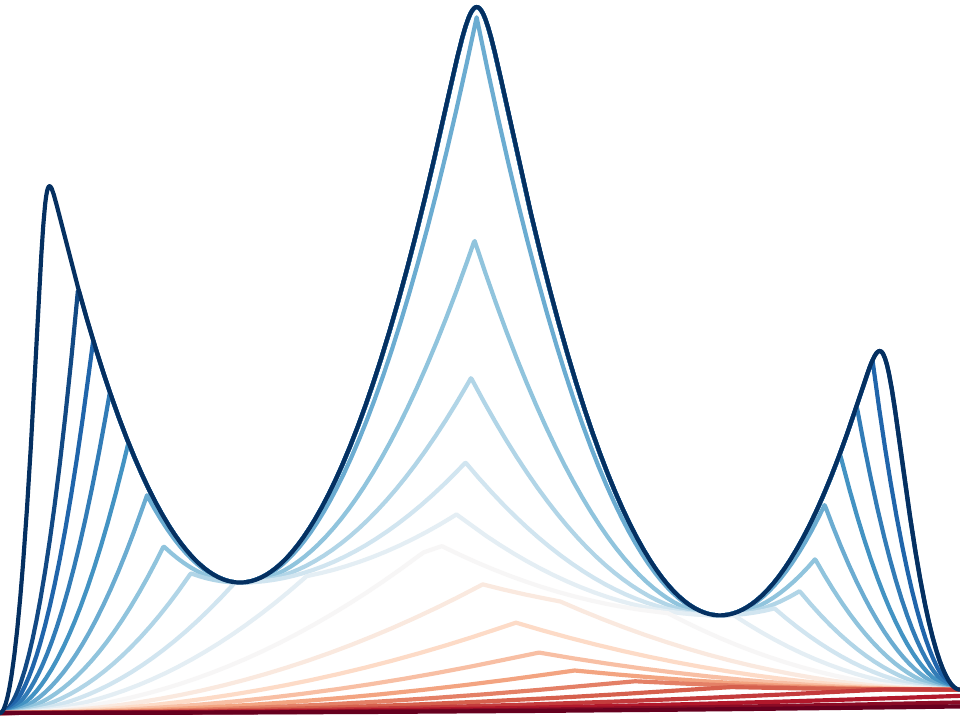}};

   \path (glas) edge[-, auto] (zoom);

    \end{tikzpicture}
    \caption{%
        \MoreauName{} envelopes of a Gaussian mixture for a sequence of \MoreauName{} parameters $t \in [\num{1e-2},\num{1e2}]$. Note how the Moreau envelope \emph{convexifies} the potential with increasing $t$.
    }
\label{fig:moreau_gm}
\end{figure}

\begin{lemma}\label{lemma:non-convexity}
    Define the \emph{nonconvexity} of a function $H:\R^d\rightarrow \R$ as the (possibly infinite) number
    \begin{equation}
        \NC(H) \coloneqq\sup_{\substack{x,y\in\R^d\\\lambda\in [0,1]}} H(\lambda x+(1-\lambda) y) - \lambda H(x) - (1-\lambda)H(y).
    \end{equation}
    Assume $\prox_{tH}(x)$ is nonempty for all $x \in \R^d$.
    Then, $\NC(M^t_H)\leq \NC(H)$.
\end{lemma}
\begin{proof}
    Let \( x, x^\prime \) be arbitrary points in \( \R^d \) and let $p\in \prox_{tH}(x)$, $p^\prime \in \prox_{tH}(x^\prime)$, $x_\lambda = \lambda x + (1-\lambda)x^\prime$, and $p_\lambda = \lambda p + (1-\lambda)p'$.
    It follows by definition of the \MoreauName{} envelope that
    \begin{equation*}
        \begin{aligned}
            M_H^t(x_\lambda)& - \lambda M_H^t(x)-(1-\lambda)M_H^t(x')\\
            &\leq  H(p_\lambda) - \lambda H(p) - (1-\lambda)H(p')
            +\tfrac{1}{2t}\|x_\lambda - p_\lambda\|^2 - \tfrac{\lambda}{2t}\|x-p\|^2 - \tfrac{1-\lambda}{2t}\|x'-p'\|^2 \\
            &\leq \NC(H).
        \end{aligned}
    \end{equation*}
    Since \( \tfrac{1}{2t}\|x_\lambda - p_\lambda\|^2 - \tfrac{\lambda}{2t}\|x-p\|^2 - \tfrac{1-\lambda}{2t}\|x'-p'\|^2 \leq 0 \) due to (strict) convexity.
    Taking the supremum over $x$, $x'$, and $\lambda$ on the left-hand side concludes the proof.
\end{proof}
\begin{remark}
    $\NC(H)=0$ if and only if $H$ is convex:
    If $\NC(H)=0$ it follows immediately that $H$ is convex.
    If, on the other hand, $H$ is convex,  $\NC(H)\leq 0$ by definition and it can be lower bounded by \( 0 \) by choosing \( x = y \).
\end{remark}
The convexifying behavior of the \MoreauName{} enables \gls{daz} to cover several modes of a multi-modal distribution more quickly.

We now give the assumptions on the potential $U = F + G$ that enable a thorough theoretical analysis of the proposed sampling algorithm.
Remarks about the assumptions are discussed immediately below.
\begin{ass}
    We make the following assumptions on the potential \( U = F + G \):
    \begin{subassumption}
        \item \label{ass:f lipschitz gradient} $F$ is differentiable and its gradient is Lipschitz continuous.
        \item \label{ass:g weakly convex} $G$ is locally bounded and weakly convex, \ie, there exists a modulus of weak convexity $\ModulusWeakConvexity > 0$ such that $G+\frac{\ModulusWeakConvexity}{2}\|\cdot\|^2$ is convex.
        \label{ass:g lipschitz or integrable} In addition, $G$ is either Lipschitz continuous or such that $\int \exp{\bigl(-G(x)\bigr)}\,\dd x<\infty$.
        \item \label{ass:convex outside ball} $F$ and $G$ are convex outside a ball, \ie, there exists an $\ConvexityRadius > 0$ such that for any $x,y \in \BallC{\ConvexityRadius}{0}$ it holds that
        \[
            G(\lambda x + (1-\lambda)y)\leq \lambda G(x) + (1-\lambda)G(y)
        \]
        and identically for $F$.
        \item \label{ass:phi integrable} Let $\tmax <\frac{1}{\ModulusWeakConvexity}$ and define $\phi(s)\coloneqq \sup \{\|p\|\;|\;p\in \partial G\bigl(\prox_{tG}(x)\bigr), \;\|x\|\leq s,\; t\leq \tmax\}$. It holds true that 
        \begin{equation}
            \int \phi(\|x\|)^2 \exp{\bigl( -U^{\tmax}(x) \bigr)}\,\dd x \eqqcolon \intbound < \infty.
            \label{eq:integrability_phi}
        \end{equation}
    \end{subassumption}
    \label{ass}
\end{ass}
\begin{remark}\label{remark:ass}
    We provide a list of implications of \Cref{ass} to facilitate the understanding of the types of potentials our theory covers.
    Moreover, some of the implications will be used within the proofs.
    \begin{enumerate}
        \item The weak convexity of $G$ is relevant thrice.
        First, it implies that \( G \) admits a regular subgradient everywhere:
        The function $\widetilde{G} \coloneqq G + \frac{\ModulusWeakConvexity}{2}\|\,\cdot\,\|^2$ admits a (convex) subgradient everywhere and by \cite[8.8 Exercise, 8.12 Proposition]{rockafellar2009variational} it follows that $\partial G (x) = \partial\widetilde{G}(x) - \ModulusWeakConvexity x$.
        Second, when $t<\frac{1}{\ModulusWeakConvexity}$ the map $y\mapsto\frac{1}{2t}\|x-y\|^2 + G(y)$ is strongly convex and, consequently, for such \( t \) the proximal map of \( G \) is at most single-valued.
        Third, it implies local boundedness of the subgradients of $G$, that is, boundedness of the set
        \[
            \bigcup\limits_{x \in \Ball{R}{0}}\partial G(x)
        \]
        for any $R>0$:
        Since \( \widetilde{G} \) is bounded from below and locally bounded from above, by \cite[9.14 Example]{rockafellar2009variational}, $\widetilde{G}$ is locally Lipschitz continuous and, consequently, $\partial \widetilde{G}$ is locally bounded by the respective Lipschitz constant.
        Local boundedness of $\partial G$ follows from $\partial G(x) = \partial\widetilde{G}(x)-\ModulusWeakConvexity x$.
        Moreover, local Lipschitz continuity of $G$ also implies \emph{global} continuity.
        Continuity of $G$ will be used frequently, \eg, in order to ensure continuity of the proximal map in \Cref{lemma:prox_time_cont}.
        Local boundedness of $\partial G$ is needed to ensure coercivity of the proximal map in \Cref{lemma:prox_coercive}.
        \item Since $G$ is bounded from below and continuous, $\prox_{tG}(x)$ is nonempty for all \( x \in \R^d \).
        \item The assumption that $G$ is either Lipschitz continuous or satisfies the integrability condition $\int \exp{\bigl(-G(x)\bigr)}\,\dd x<\infty$ ensures that $\int \exp{\bigl(-U^t(x)\bigr)}\,\dd x<\infty$, \ie, that the potential $U^t$ defines a proper Gibbs distribution, by an adaptation of \cite[Proposition 1]{durmus2018efficient} together with \Cref{lemma:U_linear_growth}.
        Without this assumption, the smoothing properties of the \MoreauName{} envelope might lead to a violation of the growth requirements on \( U^t \) that are needed for the integrability of \( x \mapsto \exp\bigl( -U^t(x) \bigr)\).\label{rmk_ass_proper}
        \item The definition of \emph{convexity outside a ball} implies also a \emph{first-order} version, that is, for any $x,y \in \BallC{\ConvexityRadius}{0}$ and $v_x\in\partial G(x)$, $v_y\in\partial G(y)$ it holds true that 
        \[\langle v_x-v_y,x-y\rangle\geq 0.\]
            The converse implication, however, is only true in a weaker sense.
            First order convexity outside the ball $\Ball{\ConvexityRadius}{0}$ implies $G(\lambda x + (1-\lambda)y)\leq \lambda G(x) + (1-\lambda) G(y)$ only whenever the entire line connecting $x$ and $y$ is outside the ball, \ie, $[x,y] \subset \BallC{\ConvexityRadius}{0}$.
        \item The last integrability assumption is technical and typically no issue in practice.
        When $G$ is Lipschitz continuous (\eg, $G$ in the $\ell^1$ norm or \gls{tv}) it is equivalent to integrability of $x \mapsto \exp{\bigl( -U^{\tmax}(x) \bigr)}$.
            Otherwise, \Cref{lemma:U_linear_growth} shows that there there exist $r, c>0$ such that $U(x) \geq c \|x\|$ for all \( x \in \BallC{r}{0} \).
            The same holds true for $U^{\tmax}$ as a consequence of \Cref{cor:moreau_convex_outside}.
            Therefore, for $x$ large enough $\exp{\bigl(-U^{\tmax}(x)\bigr)} \leq \exp{(-c\|x\|)}$, which leads to \eqref{eq:integrability_phi} being satisfied as long as $\partial G(\prox_{tG}(x))$ does not grow exponentially fast in $\|x\|$.
            In particular, if $U^{\tmax}$ is superexponential (\cf, \Cref{def:superexponential} below) it also follows that for $\|x\|$ sufficiently large $\|\prox_{tG}(x)\|\leq\|x\|+c'$ for some $c'>0$ (\Cref{lem:prox_growth_bound}). Since $G$ is locally Lipschitz, it is also almost everywhere differentiable and the constraint reduces to integrability of 
            \[
                x \mapsto \sup_{\|y\|\leq\|x\|+c'}\|\nabla G(y)\|^2\exp{(-c\|x\|)}.
            \]
    \end{enumerate}
\end{remark}

\subsection{Ergodicity of Diffusion at Absolute Zero}
In this section we investigate ergodicity and convergence properties of the \gls{ula} subroutine in \Cref{algo} for a fixed \MoreauName{} parameter \( t_n \).
This is the Markov chain
\begin{equation}\label{eq:ULA_subroutine}
    \left\{
        \begin{aligned}
            X_0 &= x_0, \\
            X_{k+1} &= X_k - \tau \nabla F(X_k) - \tfrac{\tau}{t}\left(X_k - {\prox}_{t G}(X_k) \right) + \sqrt{2\tau} Z_k,
        \end{aligned}
    \right.
\end{equation}
initialized at an arbitrary \( x_0 \in \R^d \).
We omit the subscript $n$ since the \MoreauName{} parameter is fixed.
We denote the Markov kernel that corresponds to one iteration of \eqref{eq:ULA_subroutine} as $R_\tau$, that is, if $X_0\sim \mu$, then $X_1\sim \mu R_\tau$.
This problem has been thoroughly analyzed in \cite{durmus2018efficient} for convex \( G \) and differentiable \( F \) with Lipschitz continuous gradient.
We recall a shortened version of the main result from~\cite[Section 3.2]{durmus2018efficient}.
\begin{theorem}
    Let $F$ and $G$ be lower bounded, $F$ convex and differentiable with Lipschitz continuous gradient, and $G$ proper, convex, and \gls{lsc}
    In addition, assume that \( G \) is either Lipschitz continuous or such that \( \int \exp{\bigl(-G(x)\bigr)}\,\dd x < \infty \).
    Then for any $x_0\in\R^d$ the Markov chain \eqref{eq:ULA_subroutine} is geometrically ergodic, \ie, there exists a probability measure $\pi^t_\tau$, and constants $C>0$, $\kappa\in (0,1)$ such that
    \[
        \|\delta_{x_0} R^k_\tau - \pi^t_\tau\|_{\mathrm{TV}}\leq C \kappa^k.
    \]
    Moreover, $\|\pi^t_\tau - \pi^t\|_{\mathrm{TV}}\rightarrow 0$ as $\tau\rightarrow 0$.
\end{theorem}
In addition to ergodicity, explicit and nonasymptotic convergence rates can be found in \cite{durmus2018efficient}.

In the sequel, we analyze the ergodicity of \eqref{eq:ULA_subroutine} in the nonconvex case.
We provide two main results:
The ergodicity of the Markov chain independent of the \MoreauName{} parameter \( t \) (\Cref{prop:ergodic1}) and a stronger result that proves the convergence of the chain to the continuous-time diffusion when $t<\frac{1}{\ModulusWeakConvexity}$ (\Cref{prop:erg_small_t}).
The theorems rely on several preliminary results concerning the \MoreauName{} envelope and the proximal map.
Parts of these are covered in \cite{rockafellar2009variational} but we include results that are tailored to the setting in this work in order to provide a self-contained article.
The proofs of \Cref{lemma:prox_time_cont,,lemma:Moreau_diffble_small_t,,lemma:prox_coercive,,Moreau_diffble_outside,,cor:moreau_convex_outside} are given in \Crefrange{sec:prox_time_cont}{sec:appendix_moreau_convex_outside}.
\begin{lemma}%
    \label{lemma:prox_time_cont}
    The map $(x,t)\mapsto{\prox}_{tG}(x)$ is continuous on $\R^d\times\bigl(0,\frac{1}{\ModulusWeakConvexity}\bigr)$ and, for $t\in\bigl(0,\frac{1}{\ModulusWeakConvexity}\bigr)$, $x\mapsto {\prox}_{tG}(x)$ is Lipschitz continuous with Lipschitz constant $\frac{1}{1-\ModulusWeakConvexity t}$.
\end{lemma}
\begin{lemma}%
    \label{lemma:Moreau_diffble_small_t}
    For $t\in \bigl(0,\frac{1}{\ModulusWeakConvexity}\bigr)$, $M^t_G$ is differentiable and the gradient is given by $\nabla M^t_G(x)=\frac{1}{t}\bigl(x-\prox_{t G}(x)\bigr)$.
\end{lemma}
\begin{corollary}%
    \label{lemma:moreau_nabla_lipschitz}
    For $t\in (0,\frac{1}{\ModulusWeakConvexity})$, $\nabla M^t_G$ is Lipschitz continuous with Lipschitz constant $\frac{2-\ModulusWeakConvexity t}{t(1-\ModulusWeakConvexity t)}$.
\end{corollary}
\begin{proof}
    Since $\nabla M^t_G(x) = \frac{1}{t}\bigl(x-\prox_{tG}(x)\bigr)$ the result follows from Lipschitz continuity of the proximal map.
\end{proof}
When $t\geq \frac{1}{\ModulusWeakConvexity}$, regularity properties of the \MoreauName{} envelope do not follow as directly.
However, we can obtain guarantees by using convexity of $G$ outside the ball $\Ball{\ConvexityRadius}{0}$.
To proceed, we require that $\prox_{tG}(x)$ is outside of $\Ball{\ConvexityRadius}{0}$ for sufficiently large \( x \in \R^d \), which the following lemma asserts.
\begin{lemma}%
    \label{lemma:prox_coercive}
    The proximal map is coercive in the sense that for all \( t > 0 \) it holds that
    \begin{equation}\label{eq:prox_coercive}
        \lim_{R\rightarrow\infty}\;\inf_{x \in \BallC{R}{0}}\; \inf_{p\in \prox_{tG}(x)}\|p\|=\infty.
    \end{equation}
\end{lemma}
\begin{corollary}%
    \label{Moreau_diffble_outside}
    For any $t>0$, there exists $R>0$ such that $M^t_G$ is differentiable and $\prox_{tG}$ is single-valued and 1-Lipschitz on $\BallC{R}{0}$.
    That is,
    \[
        \|\prox_{tG}(x)-\prox_{tG}(y)\|\leq \|x-y\|,\quad x,y \in \BallC{R}{0}.
    \]
    In addition, $\|\nabla M^t_G(x)-\nabla M^t_G(y)\|\leq\frac{1}{t}\|x-y\|$ for $x,y\in \BallC{R}{0}$.
    
\end{corollary}
\begin{corollary}%
    \label{cor:moreau_convex_outside}
    The \MoreauName{} envelope $M^t_G$ is convex outside a ball. That is, there exists $R>0$ such that for $x,y\not\in B_R(0)$ it holds that
    \[
        M^t_G(\lambda x + (1-\lambda)y)\leq \lambda M^t_G(x) + (1-\lambda)M^t_G(y).
    \]
\end{corollary}
Within the proof of convergence of \gls{ula} to the target density, we will make use of the following growth property \cite[Section 3]{durmus2017nonasymptotic}.
\begin{definition}[Superexponential]%
    \label{def:superexponential}
    A function $H:\R^d\rightarrow \R$ is \emph{superexponential} if there exist a minimizer $x^* \in \R^d$ of $H$ and $\rho,M_\rho>0$ such that for any $x \in \BallC{M_\rho}{x^\ast}$ and \( v\in\partial H(x) \) it holds that
    \begin{equation}\label{eq:superexp}
        \langle v,x-x^*\rangle\geq \rho\|x-x^*\|^2.
    \end{equation}
    We refer to $M_\rho$ and $\rho$ as the \emph{supex radius} and \emph{supex modulus} of $H$, respectively.
\end{definition}    
\begin{lemma}%
    \label{lemma:superexp_M}
    The \MoreauName{} envelope preserves the superexponential property.
    More specifically, if $G$ is superexponential with radius $M_\rho$ and modulus $\rho$, then there exists $M_\rho'>0$ such that for all \( x \in \BallC{M_\rho^\prime}{x^\ast} \) 
    \begin{equation}\label{eq:super_exp_moreau}
        \langle \nabla M^t_G(x),x-x^*\rangle\geq \min \Bigl(\frac{\rho}{4},\frac{1}{2t}\Bigr)\|x-x^*\|^2.
    \end{equation}
\end{lemma}
The proof is given in \Cref{sec:appendix_superexp}.
We are now in the position to prove the main results. The first results concerns ergodicity of the \gls{ula} chain, that is, existence of a unique stationary distribution and convergence of the chain to said distribution.
\begin{theorem}[Ergodicity for arbitrary $t>0$]\label{prop:ergodic1}
    Let \Cref{ass} be satisfied and assume that $G$ or $F$ is superexponential.
    Then, for a small enough $\tau>0$, the chain \eqref{eq:ULA_subroutine} is geometrically ergodic.
    That is, there exists a stationary measure $\pi^t_\tau$ and $\lambda\in (0,1)$ such that for any $x_0\in\R^d$ there exists $C>0$ with
    \[
        \|\delta_{x_0}R^k_\tau - \pi^t_\tau\|_{\mathrm{TV}}\leq C\lambda^k
    \]
    for \( k \in \N \).
\end{theorem}
\begin{proof}
    The result follows as in \cite[Theorem 5.3]{fruehwirth2024ergodicity} if we can prove that the potential $U^t$ is superexponential.
    Therefore, we now show that $F$ or $G$ being superexponential implies that $U^t$ is superexponential with appropriate radius and modulus.
    When $G$ is superexponential, \Cref{lemma:superexp_M} implies that $M^t_G$ is superexponential with appropriate radius and modulus which we denote as $M_\rho$ and $\rho$.
    Using convexity of $U^t$ outside a ball \Cref{lemma:U_linear_growth} shows that $U^t$ grows asymptotically at least linearly.
    Thus $U^t$ is coercive, bounded from below, and continuous and, consequently, admits a minimizer.
    Let $x^*_G$ be a minimizer of $G$ (and, consequently, $M^t_G$) and $x^*_U$ be a minimizer of $U^t$.
    Moreover, recall that $\ConvexityRadius$ denotes the radius of convexity of $F$ and $G$.
    Pick $M_\rho'>\max(M_\rho,\ConvexityRadius)$ such that \( x \in \BallC{M_\rho^\prime}{x_U^*} \) implies that $\|x\|>\ConvexityRadius$ and, consequently, differentiability and convexity of $M^t_G$ in a neighborhood around $x$.
    Consequently, for \( x \in \BallC{M_\rho^\prime}{x_U^*} \) it follows that
        \begin{equation}\label{eq:growth_Ut}
            \begin{aligned}
                \left\langle \nabla U^t(x), x-x_U^* \right\rangle &= \underbrace{\biggl\langle \nabla F(x)-\nabla F\biggl(M_\rho'\frac{x}{\|x\|}\biggr), x - M_\rho'\frac{x}{\|x\|} \biggr\rangle}_{\geq 0\text{ by convexity outside $\Ball{\ConvexityRadius}{0}$}}\\
                &\qquad+\biggl \langle \nabla F\biggl(M_\rho'\frac{x}{\|x\|}\biggr), x - M_\rho'\frac{x}{\|x\|} \biggl\rangle \\
                &\qquad +\biggl\langle \nabla F(x),M_\rho'\frac{x}{\|x\|}-x^*_U\biggr\rangle + \langle \nabla M^t_G(x), x-x_U^*\rangle\\
                &\geq \langle \nabla M^t_G(x), x-x_U^*\rangle - c_1\|x\| - c_2\\
                &\geq \langle \nabla M^t_G(x), x-x_G^*\rangle + \langle \nabla M^t_G(x), x_G^*-x_U^*\rangle - c_1\|x\| - c_2\\
                &\geq \rho \|x - x_U^*\|^2 - \tilde{c}_1\|x\| - \tilde{c}_2\\
            \end{aligned}
        \end{equation}
        where we used Lipschitz continuity of $\nabla F$ and $\nabla M^t_G$ to obtain the linear bounds with appropriate constants $\tilde{c}_1, \tilde{c}_2>0$.
        Thus, $U^t$ is superexponential. In the case that instead $F$ is superexponential, we can use the convexity of $M^t_G$ outside a ball and apply the same arguments.
    As in \cite[Theorem 5.3]{fruehwirth2024ergodicity} this yields the existence of a probability distribution $\pi^t_\tau$, a constant $C>0$, and $\lambda\in (0,1)$ such that
    \[
        \|\delta_{x_0}R^k_\tau - \pi^t_\tau\|_V \leq C\lambda^k
    \]
    where 
    \begin{equation}\label{eq:f_norm}
        \|\nu\|_f\coloneqq \sup_{g: |g|\leq f}\Biggl|\int g(x)\;\dd\nu(x)\Biggr|.
    \end{equation}
    and $V(x)=\exp{(\|x-x_U^*\|)}$. In particular, since for any measurable set $A$ and $x\in\R^d$ it holds that $|1_A(x)| \leq V(x)$ we obtain that
    \[
        \|\delta_{x_0}R^k_\tau - \pi^t_\tau\|_{\TV} = 
        \sup_{A\in\Bc(\R^d)}\left| \int \1_A(x)\,\dd(\delta_{x_0}R^k_\tau - \pi^t_\tau)(x) \right|\leq \|\delta_{x_0}R^k_\tau - \pi^t_\tau\|_V,
    \]
    which proves the desired convergence in \gls{tv}.
\end{proof}

When $t < \frac{1}{\ModulusWeakConvexity}$ we obtain convergence of the stationary measure $\pi^t_\tau$ to the target $\pi^t$ when the step size \( \tau \) vanishes in addition to ergodicity.
\begin{theorem}[Convergence for $t<\frac{1}{\ModulusWeakConvexity}$]%
    \label{prop:erg_small_t}
    Assume that $t<\frac{1}{\ModulusWeakConvexity}$. Let \Cref{ass} be satisfied and assume that $G$ or $F$ is superexponential.
    Then, the stationary measure $\pi^t_\tau$ of the discrete time Markov chain \eqref{eq:ULA_subroutine} approaches $\pi^t$ as $\tau\rightarrow 0$ in \gls{tv}. More precisely, for any $\epsilon>0$ there exists $\bar{\tau}>0$ and $N>0$ such that for $\tau<\bar{\tau}$ and $\|\pi^t_\tau - \pi^t\|_{\TV}<\epsilon$ and if $k>N$, $\|\delta_{x_0}R^k_\tau - \pi^t\|_{\TV}<\epsilon$ and it holds $\bar{\tau} = \mathcal{O}(-\log^{-1}(\epsilon)\epsilon^2)$ and $N=\mathcal{O}(\log^{2}(\epsilon)\epsilon^{-2})$.
\end{theorem}
\begin{proof}
    The ergodicity of the discrete chain was already proved in \Cref{prop:ergodic1}.
    To prove that $\pi^t_\tau \rightarrow \pi^t$ as \( \tau \to 0 \), we follow the three-step strategy of \cite{durmus2017nonasymptotic}.
    In the first step, we establish exponential ergodicity of the corresponding continuous-time Langevin diffusion with stationary measure $\pi^t$.
    In the second step, we combine the exponential ergodicity with the approximation of the continuous time diffusion by the discrete chain to bound $\|\delta_{x_0}R^k_\tau - \pi^t\|_{\TV}$.
    This bound can be made small as $\tau\rightarrow 0$ if the second moments of the iterates of the discrete chain are bounded uniformly in $\tau$.
    Thus, in the third step we establish such a uniform bound.

    \textbf{Step 1: Ergodicity of the continuous time process.}
    The chain \eqref{eq:ULA_subroutine} is a discretization of the Langevin \gls{sde} on the potential $U^t$ defined via 
    \begin{equation}\label{eq:erg_SDE}
    \left\{
        \begin{aligned}
            X_0 &= x_0\\
            \dd X_s &= -\nabla U^t(X_s)\dd t + \sqrt{2}\dd W_s.
        \end{aligned}
        \right.
    \end{equation}
    The existence of a unique solution for all time is a standard result when $\nabla U^t$ is Lipschitz continuous, which is guaranteed for $t<\frac{1}{\ModulusWeakConvexity}$.
    As in \cite[Theorem 2.1]{roberts1996exponential} it follows that $(X_s)_{s>0}$ is nonexplosive, irreducible with respect to the Lebesgue measure, strong Feller and, as a consequence, all compact sets are small.
    The \gls{sde} defines a Markov semi-group via $P_s f(x) = \E[f(X_s)|X_0=x]$.
    By Ito's lemma, the generator of the semi-group is given by
    \[
        \Ac V (x) = \lim_{s\rightarrow +0}\frac{\E[V(X_s)|X_0=x] - V(x)}{s} = -\langle \nabla U^t(x), \nabla V(x)\rangle + \Delta V(x)
    \]
    for any twice continuously differentiable $V : \R^d \to \R$ (\cf \cite[Lemma 4.5]{fruehwirth2024ergodicity}).
    In particular, for $V(x)=\|x-x_U^*\|^2$ the generator reads
    \begin{equation}
        \Ac V(x) = -2\langle \nabla U^t(x), x-x_U^*\rangle + 2d.
    \end{equation}
    For any $x\in \BallC{M_\rho}{x_U^*}$, the generator satisfies the drift condition
    \begin{equation}
        -2\langle \nabla U^t(x), x-x_U^*\rangle + 2d \leq -2\rho\|x-x_U^*\|^2 + 2d\leq -2\rho V(x) + 2d
    \end{equation}
    since \( U^t \) is superexponential.
    By \cite[Theorem 6.1]{meyn1993stability}, the combination of the drift condition with the fact that all compact sets are small implies the existence of $B>0$, $\kappa\in(0,1)$, and a stationary measure (which has to be the target measure $\pi^t$ as shown in \cite[Theorem 3]{fruehwirth2024ergodicity}) such that
    \begin{equation}
        \|\delta_{x_0} P_s - \pi^t\|_{V+1}\leq (\|x_0-x_U^*\|^2+1) B\kappa^s.
    \end{equation}
    It follows for any $\mu$ with bounded second moment that
    \begin{equation}\label{eq:tv_conv}
        \begin{aligned}
            \|\mu P_s - \pi^t\|_{V+1} = &\sup_{g\leq V+1}\left| \int g(x) \dd (\mu P_s - \pi^t)(x)\right| \\
            =&\sup_{g\leq V+1}\left| \int \int g(x) \dd P_s(y,\cdot)(x) \dd \mu(y) - \int g(x) \dd \pi^t(x)\right|\\
            = &\sup_{g\leq V+1}\left| \int \int g(x) \dd ( P_s(y,\cdot)-\pi^t)(x) \dd \mu(y) \right|\\
            \leq &\sup_{g\leq V+1}\left| \int (\|y-x_U^*\|^2+1) B\kappa^t \dd \mu(y) \right|
            \leq \left(\int \|y-x_U^*\|^2\,\dd\mu(y) + 1\right) B\kappa^s.
        \end{aligned}
    \end{equation}
    As previously, since \( V + 1 \) dominates $|1_A|$ for arbitrary measurable sets $A$, the inequality  also holds for the TV norm.
    
    \textbf{Step 2: Estimating the error.}
    Let us define $C:\Pc_2(\R^d)\rightarrow \R$ as $C(\mu)=(\int \|x-x_U^*\|^2\, \dd\mu(x) + 1) B$, the \emph{constant} in \eqref{eq:tv_conv}.
    Proposition 2 in \cite{durmus2017nonasymptotic} states that for any $k, n \in \N$, $k > n > 0$
    \begin{equation}\label{eq:erdogic2}
        \begin{aligned}
            \|\delta_{x_0} R^k_\tau - \pi^t\|_{\mathrm{TV}} &\leq \|\delta_{x_0} R^k_\tau - \delta_{x_0} R^n_\tau P_{(k-n)\tau}\|_{\mathrm{TV}} + \|\delta_{x_0} R^n_\tau P_{(k-n)\tau} - \pi^t\|_{\mathrm{TV}}\\
            &\leq \frac{L_{\nabla U^t}}{\sqrt{2}}\sqrt{k-n}\left(\frac{\tau^3}{3}A({x_0},\tau) + d\tau^2\right)^{1/2} + C(\delta_{x_0} R_\tau^n)\kappa^{(k-n)\tau}
        \end{aligned}
    \end{equation}
    where $A({x_0},\tau)=\sup_{k\geq 0}\int\|\nabla U^t(y)\|^2 \,\dd(\delta_{x_0} R_\tau^k)(y)$. The bound on the first term results from the discretization \eqref{eq:ULA_subroutine} of the \gls{sde} and the bound on the second term from exponential ergodicity of the continuous-time diffusion. Thus, it remains to bound the constants $A({x_0},\tau)$ and $C(\delta_{x_0} R_\tau^n)$ where it is crucial that the bounds are independent of $\tau$. By their respective definitions and Lipschitz continuity of $\nabla U^t$, it is sufficient to bound the second moments of the Markov chain $(X_k)_k$.
    
    \textbf{Step 3: Bounding the second moments of $(X_k)_k$.}
    The proof is adapted from \cite{fruehwirth2024ergodicity} and provided for the sake of completeness.
    Since \( U ^t \) is superexponential, there exists $M_\rho, \rho>0$ such that for all \( x \in \R^d \) such that $\|x-x_U^*\|\geq M_\rho$ it holds that
    \begin{equation}\label{eq:bdd_sec_mom}
        \begin{aligned}
            \|x-\tau \nabla U^t(x)-x_U^*\|^2 &\leq \|x-x_U^*\|^2 -2\tau \rho\|x-x_U^*\|^2 + \tau^2(L_{\nabla U^t}\|x-x_U^*\|)^2\\
            &\leq \|x-x_U^*\|^2\bigl(1-\tau (2\rho - \tau L_{\nabla U^t}^2)\bigr)
        \end{aligned}
    \end{equation}
    Thus, if $\tau\leq \frac{\rho}{L_{\nabla U^t}^2}$, we find that $\|x-\tau \nabla U^t(x)-x_U^*\|^2 \leq \|x-x_U^*\|^2(1-\tau\rho)$.
    Further, for \( x \in \R^d \) with $\|x-x_U^*\|< M_\rho$, the Lipschitz continuity of $\nabla U^t$ implies that
    \begin{equation}
        \begin{aligned}
            \|x-\tau \nabla U^t(x)-x_U^*\| &\leq \|x-x_U^*\| + L_{\nabla U^t} \tau \|x-x_U^*\|\\
            &\leq (1+L_{\nabla U^t} \tau)M_\rho \eqqcolon M_\tau.
        \end{aligned}
    \end{equation}
    For the iterates of our chain, this yields (by induction)
    \begin{equation*}
        \begin{aligned}
            \|X_k-x_U^*\|^2&\leq \1_{\{\|X_n-x_U^*\|>R,\; n=0,\dots,k-1\}}(1-\rho\tau)^k\|X_0-x_U^*\|^2\\
            &\qquad+ \sum\limits_{n=0}^{k-1} \left\{M_\tau^2\1_{\|X_n-x_U^*\|\leq R}\prod\limits_{\ell=n+1}^{k-1}\1_{\|X_\ell-x_U^*\|> R} \right.+(1-\rho\tau)^n\Bigl[ \\
            &\qquad 2\tau \|W_{k-1-n}\|^2 + 2\langle X_{k-1-n}^\tau-\tau\partial U(X_{k-1-n}^\tau)-x_U^*,\sqrt{2\tau}W_{k-1-n}\rangle\Bigr]\Biggr\}.
        \end{aligned}
    \end{equation*}
    Taking the expectation and noting that $W_{k-1-n}$ and $X_{k-1-n}$ are independent, it follows that
    \begin{equation*}
        \begin{aligned}
            \E\bigl[\|X_k-x_U^*\|^2\bigr]&\leq (1-\rho\tau)^k\E\bigl[\|X_0-x_U^*\|^2\bigr] + M_\tau^2\E \underbrace{\left[\sum\limits_{n=0}^{k-1}\1_{\|X_n-x_U^*\|\leq R}\prod\limits_{\ell=n+1}^{k-1}\1_{\|X_\ell-x_U^*\|> R}\right]}_{\leq 1} \\
            &\qquad+\sum\limits_{n=0}^{k-1} (1-\rho\tau)^n 2\tau \E\bigl[\|W_{k-1-n}\|^2\bigr]\\
            &\leq (1-\rho\tau)^k\E\bigl[\|X_0-x_U^*\|^2\bigr] + M_\tau^2+\sum\limits_{n=0}^{k-1} (1-\rho\tau)^n 2\tau\\
            &\leq (1-\rho\tau)^k\E\bigl[\|X_0-x_U^*\|^2\bigr] + M_\tau^2+\frac{2}{\rho},
        \end{aligned}
    \end{equation*}
    which is bounded as long as $\tau$ is bounded.
    Thus, the second moments of $(X_k)_k$ are bounded uniformly with respect to $\tau$.
    As a consequence for any $x_0$, $\sup_{\tau} A(x_0,\tau)<\infty$ and $\sup_n C(\delta_{x_0} R_\tau^n)<\infty$.
    Now fix $k-n=m$ in \eqref{eq:erdogic2} and let $k\rightarrow\infty$.
    It follows that $\|\pi^t_\tau-\pi^t\|_{\mathrm{TV}}\leq C_1\sqrt{m}\tau + C_2\kappa^{m\tau}$.
    Now for arbitrary $\epsilon>0$ define $T>0$ such that $C_2\kappa^{T-1}<\frac{\epsilon}{2}$.
    Then for $\tau<\min\Bigl(\frac{\epsilon^2}{4c_1^2 T}\Bigr)$ it follows with $m=\lfloor \frac{T}{\tau} \rfloor$ that $\|\pi^t_\tau-\pi^t\|_{\mathrm{TV}}\leq\epsilon$ which is the desired convergence and the complexity $\tau = \mathcal{O}(-\log^{-1}(\epsilon)\epsilon^2)$ and, since $k\geq m$, for the number of iterations $k=\mathcal{O}(\log^2(\epsilon)\epsilon^2)$.
\end{proof}
\begin{remark}
    We show later that the density $\pi^t$ approaches $\pi$ as $t\rightarrow 0$ in the \gls{tv} distance and, consequently, \gls{daz} constitutes a method for (approximate) sampling from nondifferentiable \emph{and} non-log-concave densities.
\end{remark}
\begin{remark}\label{rmk:step_size}
    The crucial inequality that guarantees ergodicity and enables one to bound the bias of the discretization in the previous proofs (see also \cite[Proposition 5.3]{fruehwirth2024ergodicity}) is
    \[
        \tau < \tau_{\mathrm{max}} = \frac{2\rho}{L_{\nabla U^t}^2}
    \]
    where $\rho$ is the supex modulus of $U^t$ (see, \eg, \eqref{eq:bdd_sec_mom}).
    Both, $\rho$ and $L_{\nabla U^t}$ depend on the \MoreauName{} parameter $t$, which begs the question how the step-size requirements depend on $t$. Let us, therefore, distinguish the two different Moreau parameter regimes:
    \begin{enumerate}
        \item Behavior for large $t$:
        For $U^t = F + M^t_G$, it holds that $L_{\nabla U^t} \leq L_{\nabla F} + \frac{1}{t}$ \cite[Section 3.3]{durmus2018efficient} where the upper bound saturates at $L_{\nabla F}$ for $t\rightarrow\infty$.
        However, this can be circumvented by the choice $F\equiv0$ and $G=U$, in which case \( U^t \) inherits the superexponential property directly from \( G \) due to \Cref{lemma:superexp_M}. 
        Moreover, we find that $\rho\geq \frac{1}{2t}$ when \( t \) is large enough.
        Combining this with $L_{\nabla U^t}\leq \frac{1}{t}$ from \Cref{Moreau_diffble_outside} we find that
        \begin{equation}\label{eq:step_size}
            \tau_{\mathrm{max}} = \frac{2\rho}{L_{\nabla U^t}^2}\geq \frac{\frac{1}{t}}{\frac{1}{t^2}} = t.
        \end{equation}
        Thus, the choice $F\equiv0$ and $G=U$ enables the use of increasingly large step sizes for increasingly large $t$.
        \item Behavior for small $t$:
        The above estimate on $\tau_{\mathrm{max}}$ tends to zero as $t\rightarrow 0$.
        However, if the original potential $U=G$ admits a Lipschitz gradient with Lipschitz constant $L$, we can improve the estimate.
        First, for $t\rightarrow 0$, we obtain that $\rho$ saturates due to \Cref{lemma:superexp_M}.
        Second, we can show that the Lipschitz constant of $\nabla U^t=\nabla M^t_G$ converges to that of $\nabla U=\nabla G$.
        To do so, we need to bound
        \[
            \|\nabla M^t_G(x)-\nabla M^t_G(y)\| = \frac{1}{t}\|x-\prox_{tG}(x)-(y-\prox_{tG}(y))\|.
        \]
        To this end, let us denote $q_x=x-\prox_{tG}(x)$ and $q_y=y-\prox_{tG}(y)$.
        It follows that $q_x = t\nabla G(x-q_x)$ and analogously for $y$. Therefore, $\|q_x-q_y\| \leq tL(\|x-y\| + \|q_x-q_y\|)$ and, consequently, 
        \[
            \|\nabla M^t_G(x)-\nabla M^t_G(y)\|\leq \frac{L}{1-tL}\|x-y\|.
        \]
        We find that as $t\rightarrow 0$, the Lipschitz constant \emph{saturates} at $L$ and the step size \( \tau \) can be chosen as \( \frac{\rho}{2L^2} \).
    \end{enumerate}
    Consequently, when the proximal map of the composite potential can be computed efficiently, it is advantageous to choose \( F \equiv 0 \).
    When this computation is inefficient, it is typically advantageous to split the composite potential at the cost of stronger restrictions on the step size.
\end{remark}

\subsubsection{Comparison to related works}
\label{ssec:comparison}
In the following we put the presented ergodicity results into context with respect to the literature on Langevin based sampling. The strongest assumptions on the function $G$ we have made in this paper are weak convexity, convexity outside a ball, and superexponential growth. Thus, the main contribution in comparison to the literature is that we provide theoretical results assuming neither convexity nor differentiability of the original potential $U$.

Early works on Langevin sampling such as~\cite{dalalyan2016theoretical,durmus2019high} provide theoretical guarantees under the assumption that the potential is differentiable \emph{and} convex.
In \cite{wibisono2019proximal}, guarantees are provided for distributions that satisfy the log-Sobolev inequality, which is implied by strong log-concavity.
Despite the fact that the proximal algorithm in \cite{wibisono2019proximal} can formally be applied also in the non-differentiable case, the analysis still requires a potential that is thrice continuously differentiable and has a Lipschitz continuous gradient.

Theorem 9 in \cite{durmus2017nonasymptotic} is a result about ergodicity and convergence of \gls{ula} for potentials which are superexponential but potentially nonconvex. However, the analysis assumes that the potential is twice continuously differentiable and has a Lipschitz continuous gradient. 
In our approach, for the \emph{original} potential $U$, we assume neither of these properties. For the \emph{regularized} potential $U^t$, we only require that it has a Lipschitz continuous gradient---which is guaranteed through its construction---but do not assume that it is twice continuously differentiable. This relaxation is particularly relevant for \gls{daz} as the Moreau envelope is frequently not twice continuously differentiable.
A prominent example thereof is $G = \|\,\cdot\,\|_1$, whose Moreau envelope is a sum of Huber functions that are only once continuously differentiable.
In \cite[Theorem 12]{durmus2017nonasymptotic} the authors additionally prove a separate result about ergodicity and convergence without assuming that the potential is twice differentiable, and provide more explicit constants that depend on the dimension of the space by invoking a Poincar\'e inequality.
Indeed, Theorem 12 in \cite{durmus2017nonasymptotic} is applicable for sampling from $U^t$ in the inner loop of \gls{daz} for fixed $t$.
However, the result requires a non-constant sequence of step sizes $(\tau_n)_n$ which is square summable but not summable, which prohibits the direct control of the bias (see the discussion immediately after \cite[Theorem 12]{durmus2017nonasymptotic}).
The works \cite{pereyra2016proximal,durmus2018efficient,fruehwirth2024ergodicity, habring2024subgradient,durmus2019analysis} focus on results that heavily rely on the convexity of the potential with less smoothness assumptions.
In \cite{luu2021sampling} the authors also consider potentials which are weakly convex and nondifferentiable. Like in the proposed approach, they tackle the nondifferentiability of the potential by using Moreau envelopes. However, the work does not provide any results about the ergodicity of the continuous-time \gls{sde} solution or the discretization of it. Moreover, error bounds between continuous-time \gls{sde} and discrete Markov chain are provided only in expectation and for finite time.

\subsection{Consistency of Diffusion at Absolute Zero}
The proposed approach of successively sampling from $\pi^t$ with decreasing values of $t$ only makes sense if (i) the distributions $\pi^t$ are similar for similar values of $t$ and (ii) $\pi^t\rightarrow \pi$ as $t\rightarrow 0$.
We verify these properties in this section. Recall that in the following $\tmax>0$ is fixed and satisfies $\tmax<\frac{1}{\ModulusWeakConvexity}$ (\cf\  \cref{ass:phi integrable}).
The main results will make use of the following preliminary results.
\begin{lemma}\label{lemma:moreau_diffble}
    For any $x\in\R^d$, the function $t \mapsto M_G^t(x)$ is differentiable on $(0,\tmax)$  with derivative $\partial_t M_G^t(x)=-\frac{1}{2t^2}\|x-\prox_{tG}(x)\|^2$.
\end{lemma}
The proof can be found in \Cref{sec:moreau_diffble}.
\begin{remark}
    In contrast to the standard result \cite[Theorem 5, Section 3.3.2]{evans2022partial}, we proved that $M_G^t$ satisfies the Hamilton-Jacobi equation
    \[
        \partial_tM_G^t(x) + \frac{1}{2}\|\nabla M_G^t(x)\|^2 = 0
    \]
     \emph{without} assuming Lipschitz continuity of $G$.
\end{remark}
\begin{proposition}\label{prop:Moreau_cont}
    For any $x\in\R^d$, the function $t \rightarrow M_G^t(x)$ is Lipschitz continuous on $[0,\tmax)$.
    More precisely,
    \begin{equation}
        |M_G^t(x)-M_G^s(x)|\leq \frac{|s-t|}{2}\phi(\|x\|)^2
    \end{equation}
    for $0\leq s,t<\tmax$ and $x\in\R^d$.
\end{proposition}
The proof can be found in \Cref{sec:Moreau_cont}.
\begin{corollary}\label{cor:cont_normalization}
    The map $t\mapsto Z_t \coloneqq \int \exp{\bigl( -U^t(y) \bigr)}\, \dd y$ is Lipschitz continuous on $[0,\tmax)$ with Lipschitz constant $\frac{\intbound}{2}$ where $\intbound$ is defined in \cref{ass:phi integrable}.
\end{corollary}
\begin{proof}
    Let $s,t\in [0,\tmax)$ and assume without loss of generality that $s<t$ so that $M_G^s(x)\geq M_G^t(x)$ and, consequently, $\left|1-\exp{\bigl(-M_G^s(x)-M_G^t(x)\bigr)}\right| = 1-\exp{\bigl(-M_G^s(x)-M_G^t(x)\bigr)}\leq M_G^s(x)-M_G^t(x)$.
    Using \Cref{prop:Moreau_cont} we can compute that
    \begin{equation}
        \begin{aligned}
            |Z_s-Z_t| &\leq \int \left| \exp{\bigl(-U^s(x)\bigr)}-\exp{\bigl(-U^t(x)\bigr)}\right|\,\dd x \\
            &= \int \left| 1-\exp{\bigl(-(M_G^t(x)-M_G^s(x))\bigr)}\right|\exp{\bigl(-U^s(x)\bigr)}\,\dd x\\
            &\leq \int \left| M_G^t(x)-M_G^s(x) \right|\exp{\bigl(-U^s(x)\bigr)}\,\dd x\\
            &\leq \int \frac{|s-t|}{2}\phi(\|x\|)^2 \exp{\bigl(-U^{\tmax}(x)\bigr)}\,\dd x\\
            &= \frac{\intbound}{2} |s-t|,
        \end{aligned}
    \end{equation}
    where we used \Cref{ass:phi integrable} for the last equality.
\end{proof}
\begin{proposition}\label{prop:curve_lipschitz}
    The curve $t\mapsto \pi^t$ is Lipschitz continuous on $[0,\tmax)$ with respect to the \gls{tv} norm with Lipschitz constant $\tfrac{\intbound}{2Z_0}\left( 1 + \tfrac{Z_{\tmax}}{Z_0}\right)$. If, in addition,
    \[
        \int \|x\|^p\phi(\|x\|)^2 \exp{\bigl(-U^{\tmax}(x)\bigr)}\,\dd x<\infty
    \]
    holds for any $p\in [1,\infty)$, then $t\mapsto \pi^t$ is also Hölder continuous with Hölder exponent $\frac{1}{p}$ for the Wasserstein-$p$ distance.
\end{proposition}
\begin{proof}
    Since the \gls{tv} distance can be computed as $\|\pi^s-\pi^t\|_{\TV} = \int |\pi^s(x)-\pi^t(x)|\,\dd x$ and the Wasserstein-$p$ distance for $p\in [1,\infty)$ satisfies $\Wc^p_p(\pi^s,\pi^t) \leq \int 2^{p-1}\|x\|^p|\pi^s(x)-\pi^t(x)|\,\dd x$ (see \cite[Theorem 6.15]{villani2009optimal}) we can prove both assertions by bounding $\int g(x)|\pi^s(x)-\pi^t(x)|\,\dd x$ and afterwards considering respective choices of $g$. Let $s,t\in [0,\tmax)$ and assume as before $s<t$. It follows that
    \begin{equation}
        \begin{aligned}
            &\int_{\R^d}g(x)\left|\pi^s(x)-\pi^t(x)\right|\,\dd x\\
            &\quad= \frac{1}{Z_s}\int_{\R^d}g(x)\left|\exp{\bigl(-U^s(x)\bigr)}- \exp{\bigl(-U^t(x)\bigr)}\right|\,\dd x\\
            &\qquad + \int_{\R^d}g(x)\left|\tfrac{1}{Z_s} - \tfrac{1}{Z_t}\right|\exp{\bigl(-U^t(x)\bigr)}\,\dd x\\
            &\quad\leq \frac{1}{Z_0}\int_{\R^d}g(x)\left|1- \exp{\bigl(-(M_G^t(x)-M_G^s(x))\bigr)}\right|\exp\left( -U^{\tmax}(x)\right)\,\dd x\\
            &\qquad+ \left|\tfrac{1}{Z_s} - \tfrac{1}{Z_t}\right|Z_{\tmax}\int_{\R^d}g(x)\pi^{\tmax}(x)\,\dd x\\
            &\quad\leq \tfrac{1}{2Z_0}|s-t|\int_{\R^d}g(x)\phi(\|x\|)^2\exp\left( -U^{\tmax}(x)\right)\,\dd x\\
            &\qquad + \left|\tfrac{Z_t-Z_s}{Z_0^2}\right|Z_{\tmax} \E_{X\sim\pi^{\tmax}}[g(X)] \\
            &\quad\leq \tfrac{\intbound}{2Z_0}\left( 1 + \frac{Z_{\tmax}}{Z_0}\E_{X\sim\pi^{\tmax}}[g(X)]\right)|s-t|
        \end{aligned}
    \end{equation}
    where we made again use of \Cref{ass:phi integrable}.
    The choice $g \equiv 1$ proves Lipschitz continuity with respect to the \gls{tv} norm with the stated Lipschitz constant.
    The choice $g(x) = 2^{p-1}\|x\|^p$ proves Hölder continuity with exponent $\frac{1}{p}$ with respect to the Wasserstein-$p$ distance under the additional assumption.
\end{proof}
\begin{remark}
    The estimation of the Lipschitz constant in \Cref{prop:curve_lipschitz} can inform the choice of the sequence of Moreau parameters \( t_N,t_{N-1},\dotsc,t_1 \) since it can be used to bound the \gls{tv} distance between the consecutive distributions \( \pi^{t_{n+1}} \) and \( \pi^{t_{n}} \).
    While the Lipschitz constant is difficult to estimate without further assumptions on $U$, one may obtain upper bounds using, \eg, \cref{lemma:U_linear_growth} and, if available, growth bounds on the subgradient of $G$.
    As an illustrative example, consider the Gaussian $G(x) = \tfrac{\|x\|^2}{2}$. Then, $\pi^t(x) \propto \exp \bigl(-\tfrac{\|x\|^2}{2(1+t)}\bigr)$ and $Z_0 = (2\pi)^{d/2}$ and $Z_{t} = (2\pi (1+t)^2)^{d/2}$. Moreover, one can easily check that $\intbound = Z_{\tmax}$. Thus, the Lipschitz constant of $t\mapsto\pi^t$ is
    \[
        \frac{\intbound}{2Z_0}\left( 1 + \frac{Z_{\tmax}}{Z_0}\right) = \frac{1}{2}(1+\tmax)^{d}(1+(1+\tmax)^{d}) = \frac{1}{2}(1+\tmax)^{d} +(1+\tmax)^{2d}.
    \]
    In particular, the Lipschitz constant scales exponentially with respect to the dimension and polynomially with respect to $\tmax$. We want to point out, however, that the proposed scheme will sample correctly for any sequence of Moreau parameters $(t_n)_n$ as long as the final Moreau parameter is sufficiently small since the inner loop in \cref{algo} is ergodic with stationary distribution close to the target (for small $t$ and appropriate $\tau$). The annealing scheme for DAZ is solely an acceleration so that the exponential scaling of the Moreau levels may be ignored without the risk of losing the correct convergence of the scheme.
\end{remark}

As a numerical confirmation of \cref{prop:curve_lipschitz} we show in \cref{fig:compare_moreau} the estimated \gls{tv} distances $\|\pi^t_\tau-\pi\|_{\mathrm{TV}}$ as well as $\|\pi^t-\pi\|_{\mathrm{TV}}$ in the case of a Laplace distribution\footnote{While $U(x) = |x|$ is not superexponential, ergodicity and convergence of \gls{ula} applied to the potential $M^t_G$ in this case follow from \cite[Section 3.2]{durmus2017nonasymptotic}.} $U(x) = G(x) = |x|$ for $x\in\R$ for different values of the Moreau parameter $t \in \{0.001,0.005,0.01,0.05,0.1,0.5,1.0 \}$. The distance $\|\pi^t-\pi\|_{\mathrm{TV}}$ is computed by numerical integration over $[-10,10]$. The distance $\|\pi^t_\tau-\pi\|_{\mathrm{TV}}$ is estimated by sampling \num{10000} independent chains for \num{2000} iterations using \gls{ula} with the potential $M^t_G$ and step size $\tau = \frac{t}{2}$ in accordance with \eqref{eq:step_size} and afterwards computing the TV distance between $\pi$ and the empirical sample distribution, again, using numerical integration. The theoretical Lipschitz continuity of $t\mapsto\pi^t$ with respect to the \gls{tv} norm is also reflected in the practical experiments. Moreover, the errors $\|\pi^t_\tau-\pi\|_{\mathrm{TV}}$ increase as the step size $\tau$ increases.

\begin{figure}
\centering
    \input{tikz_code/compare_moreau.tikz}
    \caption{\Gls{tv} distance between the target $\pi(x)\propto \exp(-G(x))$ with $G(x) = |x|$ and the Moreau envelope based potentials $\pi^t(x)\propto \exp(-M^t_G(x))$ as well as the corresponding stationary distributions of the \gls{ula} chain $\pi^t_\tau$ for step sizes $\tau\in\{0.001,0.1,0.5\}$. We clearly observe the proven Lipschitz continuity of $t\mapsto\pi^t$ with $\pi^t\rightarrow \pi$ as $t\rightarrow 0$.
    }
\label{fig:compare_moreau}
\end{figure}


\subsection{Relation to diffusion models}
The practical differences between the proposed approach and annealed Langevin sampling---namely, that the proposed approach is applicable to any given potential (so long as it satisfies our assumptions) without training---were already discussed in~\cref{sec:related}.
In this section, we establish a theoretical relation between the proposed method and annealed Langevin sampling by providing convergence results of the potential used in annealed Langevin sampling to that used in \gls{daz} through a zero-temperature limit.
In the following we consider for simplicity the setting $U=G$ and $F\equiv 0$.

As elaborated in \cref{sec:related}, one variant of annealed Langevin sampling is to consider a family of distributions indexed by \( t \) and defined as $\pi*\Nc(0,\sqrt{t})$.
The corresponding log-density reads as
\begin{equation}\label{eq:lse_motivation}
     \log(\pi*\Nc(0,\sqrt{t})) = \log\left(\frac{1}{(2\pi t)^{d/2}}\int_{\R^d} \exp\left(-G(y) - \frac{1}{2t}\|\,\cdot\,-y\|^2\right)\,\dd y\right) + const.
\end{equation}
where the (unknown) $const.$ ensures normalization.
A major inspiration for \gls{daz} is the observation that the finite-sum version of this log-integral-exp expression is frequently encountered in machine learning and related fields and is also referred to as a \emph{softmax}.
One way to interpret the present work is that this softmax is replaced by a strict maximum, \ie,
\begin{equation}
     \log(\pi*\Nc(0,\sqrt{t})) \approx \frac{1}{(2\pi t)^{d/2}}\max_{y \in \R^d} -G(y) - \frac{1}{2t}\|\,\cdot\,-y\|^2 = -\frac{1}{(2\pi t)^{d/2}}M^t_G,
\end{equation}
where we obtain precisely the Moreau envelope up to a multiplicative factor.

In this section, we formalize the relation between the softmax potential and the Moreau potential via the introduction of a (Boltzmann) temperature \( T \).
For the finite-sum case, it is well known\footnote{%
    This can be derived in various ways; a particularly cute one is to identify that the convex conjugate of \( l(x) = \log(\sum_{i=1}^k \exp(x_i)) \) is the Shannon entropy restricted to the simplex \( \triangle^k \), that is \( l^*(y) = \sum_{i=1}^k y_i\log(y_i) + \iota_{\triangle^k}(y) \), see~\cite[Example 3.25]{Boyd_Vandenberghe_2004}.
    Since \( l_T \coloneqq Tl(\,\cdot\,/T) \) is closed and convex, \( l_T = l_T^{**} \), and standard scaling laws imply that \( l_T^* = Tl^* \).
    Thus, \( l_T = l_T^{**} = \sup_{y\in\triangle^k} \langle \,\cdot\,, y \rangle - Tl^*(y) \), which is the standard maximum when \( T = 0 \).
} that for any \( x \in \R^k \),
\begin{equation}
    \max_{i\in\{1,\dotsc,k \}} x_i = \lim_{T \to 0} T \log\Biggl( \sum_{i=1}^k \exp\biggl(\frac{x_i}{T}\biggr) \Biggr).
\end{equation}
In analogy, we might expect that
\[-M^t_G(y) = \lim_{T\rightarrow 0} T\log\left(\int \exp\left(- \frac{1}{T}(G(x) + \frac{1}{2t}\|x-y\|^2)\right)\dd x\right).\]
We will now establish this convergence in a rigorous manner, pointwise for the potential and its gradient as well as in total variation. To do so, motivated by the above, we define the following unnormalized diffusion-based potential with temperature $T>0$
\begin{equation}\label{eq:diff_pot}
    G^t_T(x) = -T\log\left(\frac{1}{(2\pi T t)^{d/2}}\int_{\R^d} \exp\left(-\frac{G(y)}{T} - \frac{1}{2T t}\|x-y\|^2\right)\,\dd y\right).
\end{equation}
If the temperature $T$ is set to \num{1}, this is precisely the negative log-density of the distribution of variance exploding diffusion \eqref{eq:lse_motivation}. By letting $T\rightarrow 0$, on the other hand, we recover the \gls{daz} potential as will be shown below.

An alternative motivation for the consideration of the potential $G^t_T$ arises through an analysis of the Hamilton-Jacobi equation:
As shown above, the Moreau envelope satisfies
\begin{equation}
    \begin{cases}
        \partial_t M^t_G (x) + \frac{1}{2}\|\nabla M^t_G (x)\|^2 = 0,\quad t>0\\
        M^0_G = G.
    \end{cases}
\end{equation}
As an approximation, one might consider for small $T>0$ the function $G^t_T$ that solves
\begin{equation}\label{eq:approx_HJ}
    \begin{cases}
        \partial_t G^t_T (x) + \frac{1}{2}\|\nabla G^t_T (x)\|^2 = T \Delta G_T^t,\quad t>0\\
        G^0_T = G
    \end{cases}
\end{equation}
whose solution converges to $M^t_G$ uniformly as $T\rightarrow 0$ if $G$ is bounded and Lipschitz\footnote{Neither of which is assumed in this manuscript.} \cite[Theorem 5.1]{crandal1984two}. The function $G^t_T$ from \eqref{eq:diff_pot} is, indeed, a solution to \eqref{eq:approx_HJ}. Based on these observations, in \cite{heaton2024global,osher2023hamilton} the authors propose to estimate the Moreau envelope by computing \eqref{eq:diff_pot} for small $T>0$, which is done via Monte Carlo integration. 

The following proposition formally states a result about the convergence of the two potentials.
\begin{proposition}\label{prop:diffusion_potential}
    The \gls{daz} potential $M^t_G$ is obtained as the zero-temperature limit of the diffusion potential $G^t_T$ from \eqref{eq:diff_pot}.
    That is, $G^t_T(x)\rightarrow M_G^t(x)$, as $T\rightarrow 0^+$ for any $x \in \R^d$.
    In addition, if $G$ is differentiable and $\nabla G$ is locally Lipschitz continuous, the convergence is uniform on compact sets.
\end{proposition}
\begin{proof}
    The result is an extension of the convergence $\|\,\cdot\,\|_p\rightarrow\|\,\cdot\,\|_\infty$ as $p\rightarrow\infty$. Using Hölder's inequality, we find that
    \begin{equation}
        \begin{aligned}
            &\frac{1}{(2\pi T t)^{d/2}}\int_{\R^d} \exp\left(\frac{-G(y)-\frac{1}{2t}\|x-y\|^2}{T}\right)\,\dd y\\
            &\quad= \frac{1}{(2\pi T t)^{d/2}}\int_{\R^d} \exp\left(-G(y)-\frac{1}{2t}\|x-y\|^2\right)^{\frac{1}{T}-1}\exp\left(-G(y)-\frac{1}{2t}\|x-y\|^2\right)\,\dd y\\
            &\quad\leq \exp\left(-M_G^t(x)\right)^{\frac{1}{T}-1} \frac{1}{(2\pi T t)^{d/2}}\int_{\R^d}\exp\left(-G(y)-\frac{1}{2t}\|x-y\|^2\right)\,\dd y\\
            &\quad\leq \exp\left(-M_G^t(x)\right)^{\frac{1}{T}-1} \frac{1}{(2\pi T t)^{d/2}}\int_{\R^d}\exp\left(-\frac{1}{2t}\|x-y\|^2\right)\,\dd y\\
            &\quad= \exp\left(-M_G^t(x)\right)^{\frac{1}{T}-1}\frac{1}{T^{d/2}},
        \end{aligned}
    \end{equation}
    where we used that \( G \geq 0\).
    Taking the negative logarithm on both sides and multiplying by $T>0$ yields
    \begin{equation}\label{eq:diffusion1}
        \begin{aligned}
            G^t_T(x)\geq (1-T)M_G^t(x) + \tfrac{d}{2}T\log(T)
        \end{aligned}
    \end{equation}
    Letting $T\rightarrow0$ proves that $\liminf_{T\rightarrow 0}G^t_T(x)\geq M_G^t(x)$.
    For the opposite bound, let us consider $\Omega_\epsilon(x)\coloneqq \left\{ y\in\R^d\;\middle|\; G(y) + \frac{1}{2t}\|x-y\|^2\leq M_G^t(x) + \epsilon\right\}$ for some small \( \epsilon > 0 \).
    The continuity of $G$ implies that this set has nonzero Lebesgue measure that we denote by \( |\Omega_\epsilon(x)| \).
    Moreover,
    \begin{equation}
        \begin{aligned}
            \frac{1}{(2\pi T t)^{d/2}}\int_{\R^d} \exp\left(\frac{-G(y) - \frac{1}{2t}\|x-y\|^2}{T}\right)\,\dd y
            \geq \frac{1}{(2\pi T t)^{d/2}}|\Omega_\epsilon(x)|\exp\left(-\frac{M_G^t(x)+\epsilon}{T}\right)
        \end{aligned}
    \end{equation}
    which shows that $|\Omega_\epsilon(x)|<\infty$. Taking the negative logarithm and multiplying by $T>0$ again, we find
    \begin{equation}\label{eq:diffusion2}
        \begin{aligned}
            G^t_T(x)\leq - T\log(|\Omega_\epsilon(x)|) + Td/2\log(2\pi T t) + M_G^t(x) + \epsilon.
        \end{aligned}
    \end{equation}
    Since $\epsilon$ was arbitrary, we conclude that $\limsup_{T\rightarrow 0}G^t_T(x)\leq M_G^t(x)$ which shows the pointwise convergence.
    
    To obtain uniform convergence on compact sets under the additional assumptions, we first note that if $x\in K$ with $K$ compact, $G(x)$ and $M^t_G(x)$ are bounded and, consequently, $T$ can be chosen uniformly for $x\in K$ in \eqref{eq:diffusion1}.
    For the upper bound we need to estimate $\log(|\Omega_\epsilon(x)|)$.
    By local Lipschitz continuity, there exists $L>0$ such that $\nabla G$ is $L$-Lipschitz on the compact set $K$.
    The first-order approximation around $p\in \prox_{tG}(x)$ combined with the descent lemma shows that
    \begin{equation}
        \begin{aligned}
            G(y) + \frac{1}{2t}\|y-x\|^2 
            \leq M_G^t(x) + \biggl(\frac{L}{2}+\frac{1}{2t}\biggr)\|y-p\|^2.
        \end{aligned}
    \end{equation}
    Consequently, $\left\{ y\;\middle| (\frac{L}{2}+\frac{1}{2t})\|y-p\|^2\leq \epsilon \right\}\subset \Omega_\epsilon(x)$. Thus, we can lower bound $|\Omega_\epsilon(x)|$ uniformly for $x\in K$, which allows us also to choose $T$ uniformly for $x\in K$ in \eqref{eq:diffusion2} concluding the proof.
\end{proof}

\begin{proposition}\label{prop:zero_temp_TV}
    If $\nabla G$ is globally Lipschitz then, the distribution of annealed Langevin sampling converges to that of \gls{daz} in \gls{tv}. More precisely, we have for the probability densities $\pi^t_T(x) \propto \exp\left(-G^t_T(x)\right)$ and $\pi^t(x)\propto \exp\left(-M^t_G(x)\right)$ for any $t>0$ that
    \[
    \lim_{T\rightarrow 0}\|\pi^t_T - \pi^t\|_{\mathrm{TV}}=0.
    \]
\end{proposition}
\begin{proof}
    If $\nabla G$ is globally Lipschitz, then by the same arguments as in \cref{prop:diffusion_potential} we find that $|\Omega_\epsilon| \geq c \epsilon^{d/2}$ for some dimension-dependent constant $c$.\footnote{Of course, $c$ is explicit. The exact form is, however, not relevant for our purposes.} It follows that
    \begin{equation}
        \begin{aligned}
            &(1-T)M^t_G(x)+\tfrac{d}{2}T\log(T) \\&\quad\leq G^t_T(x) \\&\quad\leq  - T\log(|\Omega_\epsilon(x)|) + T d/2 \log(2\pi T t) + M_G^t(x) + \epsilon\\
            &\quad\leq - T\log(c\epsilon^{d/2}) + Td/2\log(2\pi T t) + M_G^t(x) + \epsilon.
        \end{aligned}
    \end{equation}
    Thus, we can estimate
    \begin{equation}\label{eq:diffusion_lebesgue}
        \begin{aligned}
            |G^t_T(x)-M_G^t(x)|&\leq \max\left\{TM_G^t(x)+\tfrac{d}{2}T|\log(T)|, T|\log(c\epsilon^{d/2})| + Td/2|\log(2\pi T t)| + \epsilon\right\}\\
            &\leq TM_G^t(x) - c T(\log(T\epsilon)-1) + \epsilon,
        \end{aligned}
    \end{equation}
    where $c>0$ is an appropriate constant and assuming without loss of generality $T,\epsilon<1$. We begin by showing the convergence
    \begin{equation}\label{eq:tv_conv1}
    \int_{\R^d} |\exp\left(-G^t_T(x)\right)-\exp\left(-M^t_G(x)\right)|\dd x\rightarrow 0
    \end{equation}
    as $T\rightarrow 0$ for which we use Lebesgue's dominated convergence theorem. The integrand converges to zero pointwise by \cref{prop:diffusion_potential}. Thus, we are left to show that there exists a nonnegative integrable upper bound. We can estimate
    \begin{equation}\label{eq:diff_lebesgue2}
        \begin{aligned}
            \left|\exp\left(-G^t_T(x)\right)-\exp\left(-M^t_G(x)\right)\right| &= |\exp\left(M^t_G(x)-G^t_T(x)\right)-1|\exp\left(-M^t_G(x)\right)\\
            &\leq \left(\exp\left(|M^t_G(x)-G^t_T(x)|\right)+1\right)\exp\left(-M^t_G(x)\right).
        \end{aligned}
    \end{equation}
    Inserting \eqref{eq:diffusion_lebesgue} yields
    \begin{equation}
        \begin{aligned}
        &\left|\exp\left(-G^t_T(x)\right)-\exp\left(-M^t_G(x)\right)\right| \\
        &\quad\leq \left( \exp\left( TM_G^t(x) - c T(\log(T\epsilon)-1) + \epsilon\right) + 1\right) \exp\left(-M^t_G(x)\right)\\
        &\quad= \exp\left( -(1-T)M_G^t(x) - c T(\log(T\epsilon)-1) + \epsilon\right) + \exp\left(-M^t_G(x)\right).
        \end{aligned}
    \end{equation}
    Since we are interested in the convergence for $T\rightarrow 0$, we can without loss of generality assume $T\in(0,1/2]$ for which the above admits an integrable upper bound by \cref{lemma:U_linear_growth} together with the fact that $\int \exp\left(- M^t_G(x)\right)\dd x<\infty$ as elaborated in \cref{remark:ass}, \cref{rmk_ass_proper}. Thus, we have proven \eqref{eq:tv_conv1}. Let us now denote the respective partition functions as
    \[
    Z(T) = \int_{\R^d} \exp\left(-G^t_T(x)\right)\dd x, \quad Z = \int_{\R^d} \exp\left(-M^t_G(x)\right)\dd x.
    \]
    We find that
    \begin{equation}
        \begin{aligned}
            \|\pi^t_T - \pi^t\|_{\mathrm{TV}} &= \int_{\R^d} \left| \frac{1}{Z(T)}\exp\left(-G^t_T(x)\right) - \frac{1}{Z}\exp\left(-M^t_G(x)\right) \right|\dd x \\
            &= \frac{1}{Z} \int_{\R^d} \left| \exp\left(-G^t_T(x)\right) - \exp\left(-M^t_G(x)\right) \right|\dd x \\
            &\qquad + \int_{\R^d} \left| \frac{1}{Z(T)} - \frac{1}{Z}\right|\exp\left(-G^t_T(x)\right)\dd x\\
            &= \frac{1}{Z} \int_{\R^d} \left| \exp\left(-G^t_T(x)\right) - \exp\left(-M^t_G(x)\right) \right|\dd x + 
            \frac{\left|Z(T)-Z\right|}{Z}\\
        \end{aligned}
    \end{equation}
    which tends to zero as $T\rightarrow 0$ by \eqref{eq:tv_conv1} as shown in the first part of the proof.
\end{proof}

\begin{proposition}\label{prop:diffusion_gradient}
    The \gls{daz} score is obtained as the zero-temperature limit of the diffusion score.
    That is, $\nabla G^t_T(x)\rightarrow \nabla M_G^t(x)$, as $T\rightarrow 0^+$ for any $x\in\R^d$.
    In addition, if $G$ is differentiable and $\nabla G$ is locally Lipschitz continuous, the convergence is uniform on compact sets.
\end{proposition}
\begin{proof}
    Using Lebesgue's dominated convergence theorem to swap integration and differentiation we obtain that
    \begin{equation*}
        \begin{aligned}
            \frac{\partial}{\partial x_i} \int_{\R^d} \exp\left(\frac{-G(y) - \frac{1}{2t}\|x-y\|^2}{T}\right)\,\dd y = -\int_{\R^d} \frac{1}{Tt}(x_i-y_i)\exp\left(\frac{-G(y) - \frac{1}{2t}\|x-y\|^2}{T}\right)\,\dd y.
        \end{aligned}
    \end{equation*}
    By the chain rule it follows that
    \begin{equation}
        \begin{aligned}
            \nabla G^t_T(x) &=
            \int_{\R^d} \frac{1}{t}(x-y)\underbrace{\left( \frac{\exp\left(\frac{-G(y) - \frac{1}{2t}\|x-y\|^2}{T}\right)}{\int_{\R^d} \exp\left(\frac{-G(z) - \frac{1}{2t}\|x-z\|^2}{T}\right)dz}\right)}_{\eqqcolon\rho_T(y|x)}\,\dd y\\
            &= \tfrac{1}{t}(x-\E_{Y\sim\rho_T(y|x)}[Y])
        \end{aligned}
    \end{equation}
    Denoting $p={\prox}_{tG}(x)$ it follows for the difference between the \gls{daz} score and the diffusion score that for any \( \nu > 0 \)
    \begin{equation}\label{eq:diffusion_gradient}
        \begin{aligned}
            \|\nabla M^t_G(x) - \nabla G^t_T(x)\| &\leq
            \int_{\R^d} \tfrac{1}{t}\|y-p\|\rho_T(y|x)\,\dd y\\
            &= \int_{\Ball{\nu}{p}} \tfrac{1}{t}\|y-p\|\rho_T(y|x)\,\dd y + \int_{\BallC{\nu}{p}} \tfrac{1}{t}\|y-p\|\rho_T(y|x)\,\dd y\\
            &\leq \frac{\nu}{t} + \int_{\BallC{\nu}{p}} \tfrac{1}{t}\|y-p\|\rho_T(y|x)\,\dd y.
        \end{aligned}
    \end{equation}
    To show that the above tends to zero we have to show that the density $\rho_T(y|x)$ approaches a Dirac measure concentrated at ${\prox}_{tG}(x)$ \emph{quickly enough} as $T\rightarrow 0$.
    For a small enough $t>0$, weak convexity of $G$ implies that $y\mapsto G(y) + \frac{1}{2t}\|x-y\|^2$ is strongly convex and, consequently, that there exists a constant $C>0$ such that $y\mapsto G(y) + \frac{1}{2t}\|x-y\|^2 - M^t_G(x)\geq C\|y-p\|^2$.
    As mentioned before, for any $\epsilon>0$ the set $\Omega_\epsilon(x)\coloneqq \left\{ y\in\R^d\;\middle|\; G(y) + \frac{1}{2t}\|x-y\|^2\leq M_G^t(x) + \epsilon\right\}$ has positive Lebesgue measure and we can compute that
    \begin{equation}
        \begin{aligned}
            \frac{\int_{\BallC{\nu}{p}}\|y-p\|\exp\left(\frac{-G(y) - \frac{1}{2t}\|x-y\|^2}{T}\right)\,\dd y}{\int_{\R^d}\exp\left(\frac{-G(y) - \frac{1}{2t}\|x-y\|}{T}\right)\,\dd y}
            &\leq \frac{\int_{\BallC{\nu}{p}}\exp\left(\frac{-C\|y-p\|^2-M_G^t(x)}{T}\right)\,\dd y}{|\Omega_\epsilon(x)|\exp(-\frac{M_G^t(x)+\epsilon}{T})}\\
            &= \frac{1}{|\Omega_\epsilon(x)|}\int_{\BallC{\nu}{p}}\|y-p\|\exp\left(\frac{-C\|y-p\|^2+\epsilon}{T}\right)\,\dd y
        \end{aligned}
    \end{equation}
    Picking $\epsilon<C\nu^2$ and using Lebesgue's dominated convergence theorem we find that as $T\rightarrow 0$ the above integral tends to zero for any $\nu>0$.
    Therefore by \eqref{eq:diffusion_gradient},
    \[
        \limsup_{T\rightarrow 0}|\nabla M^t_G(x) - \nabla G^t_T(x)|\leq \frac{\nu}{t}.
    \]
    Since $\nu>0$ was arbitrary, this concludes the proof. Uniform convergence on compact sets in the case that $\nabla G$ is locally Lipschitz is obtained by the same arguments as above.
\end{proof}
Within the above proof we have made use of the well-known Tweedie formula $\nabla G^t_T(x) =\tfrac{1}{t}(x-\E_{Y\sim\rho_T(y|x)}[Y])$, which relates the gradient of the potential with the MMSE for denoising. By viewing the proximal operator as a MAP denoiser, the result effectively states that the \emph{MMSE-Tweedie} formula converges to the Moreau gradient formula $\nabla M^t_G(x) =\tfrac{1}{t}(x-\prox_{tG}(x))$, which constitutes a \emph{MAP-version} of Tweedie.


\section{Numerical Experiments}%
\label{sec:experiments}
In this section, we demonstrate the advantages of the proposed sampling algorithm over current algorithms on an extensive set of numerical experiments.
We begin with  a one-dimensional example that allows for the efficient computation of error metrics between the sample distribution and the reference distribution.
Then, we consider the high-dimensional examples of \gls{tv} prior sampling, \gls{tv}-L2 denoising on a chain, \gls{tv}-L2 denoising on images, sampling from the \gls{tdv} prior \cite{kobler2020total}, and \gls{mri} reconstruction using the energy-based prior proposed in \cite{zach2023stable}.
In all experiments, we compare several different sampling methods:
\gls{ula} \cite{roberts1996exponential}, \gls{myula} \cite{durmus2018efficient}, \gls{rock} \cite{pereyra2020skrock}, \gls{ula} with decreasing step sizes which we refer to as \gls{ald}\footnote{%
    This constitutes a small abuse of terminology since the potential remains unchanged as the step size changes, which is contrary to the use of this terminology in the early papers on diffusion models.
}, the proposed sampling algorithm \gls{daz}, and \gls{daz}-\gls{rock}.
In the latter, we use \gls{rock} instead of basic \gls{ula} for the inner loop in \cref{algo}.
We use subgradient steps in the updates of \gls{ula} and \gls{ald} when the potential is nondifferentiable.

The \MoreauName{} parameters and step sizes will always be chosen as follows:
The sequence of \MoreauName{} parameters $(t_n)_{N \geq n\geq 1}$ is determined by the specification of the endpoints \( t_N \) and \( t_1 \) and the loglinear computation of the intermediate values as
\begin{equation}
    t_n = 10^{\frac{n-1}{N-1}\log_{10}\frac{t_N}{t_1} + \log_{10}(t_1)}, \quad n=1,\dots, N.
\end{equation}
The endpoints are specified in the respective section of each experiment.
The sequence of step sizes $(\tau_n)_n$ is derived from $(t_n)_{N \geq n\geq 1}$ such that it satisfies the step size requirements discussed in \Cref{rmk:step_size}.
For \gls{ula}, we use the final step size $\tau_1$ across all iterations.
For \gls{myula}, we use the final Moreau parameter and step size, $t_1$ and $\tau_1$, respectively, across all iterations.
For \gls{rock} we fix the parameters $\eta=0.05$ and $s=5$ \cite[Algorithm 3.1]{pereyra2020skrock}. The Moreau parameter for \gls{rock} is chosen as the final (smallest) Moreau parameter $t_1$ with the step size for \gls{rock} as $0.9\times \delta_s^\mathrm{max}$ with $\delta_s^\mathrm{max}$ according to \cite[Algorithm 3.1]{pereyra2020skrock}.
For \gls{ald}, we use the same step sizes for each iteration as for \gls{daz}.
However, we want to emphasize that \gls{ald} has the same update rule as \gls{ula} directly applied to $\pi\propto\exp\left(-U(x)\right)$ and the step size requirements for ergodicity of \gls{ula} are stricter than those of \gls{daz} (\cf~\cref{rmk:step_size}). Therefore, especially at the beginning of the iterations where the step sizes are largest it might be the case that \gls{ald} violates these step-size requirements.
For \gls{daz}-\gls{rock} we use the same Moreau parameter scheme as for \gls{daz} and choose again always the largest feasible step size according to \cite[Algorithm 3.1]{pereyra2020skrock}. For methods involving \gls{rock} we count each update of the method as $s=5$ iterations in order to provide a fair comparison.\footnote{The parameter $s$ in \gls{rock} constitutes the order of the scheme so that one update within \gls{rock} is comparable to $s$ (sub-)gradient or proximal updates.}

In all experiments, we denote the reference distribution with \( \pi \) and the sample distributions after \( k \) iterations of the various methods with \( \pi_k \).
More specifically, in each experiment we run many parallel Markov chains and at any iteration \( k \), \( \pi_k \) is the sample distribution across all those parallel chains.
For the experiments in \Cref{sec:GM,sec:TV_chain,sec:TV_image}, we simulate \num{1000} parallel chains. For the \gls{tv} prior experiment (\Cref{sec:tv_marginal}), we simulate \num{100000} parallel chains, and for the experiment that involves costly neural-network evaluations (\Cref{sec:MRI}), we simulate \num{500} parallel chains due to computational limitations.
For the prior sampling experiment in \Cref{sec:TDV} we only compute a single Markov chain since we do not compute any statistics.
We define one iteration as one update step of \Cref{alg:update} in \Cref{algo}, \ie, one evaluation of the gradient of the Moreau envelope.
Hence, this does not account for any inner iterations when the proximal map of \( G \) is computed iteratively.
When the experiments involve nonconvex potentials, which is the case in \Cref{sec:TDV,sec:MRI}, the computations of the proximal maps might yield local minima.

To cover a range of design choices that the proposed sampling algorithm allows, we choose the number of iterations per \MoreauName{} level as \( K = 1 \) (which resembles diffusion models) in \Cref{sec:tv_marginal}, \(K=20\) in \Cref{sec:GM,sec:TV_chain,sec:TV_image}, and \( K = 200 \) (which resembles annealed Langevin sampling as in \cite{song2019generative}) in \Cref{sec:TDV,sec:MRI}.
\subsection{One-dimensional Gaussian mixture}%
\label{sec:GM}
We consider a multi-modal (and, consequently, not log-concave) one-dimensional Gaussian mixture $\pi = \sum_{i=1}^4 w_i \Nc( \mu_i,\sigma_i)$.
The weights and parameters of the individual Gaussians are chosen as $w = (0.2,0.2,0.3,0.3)$, $\mu = (-2,-1,1,2)$, and $\sigma = (0.05,0.25,0.25, 0.1)$. (These parameters were also used to produce \Cref{fig:moreau_gm}.)
This potential satisfies \Cref{ass} and is superexponential and, consequently, the presented theory applies.
In \Cref{fig:mixture_potentials} we show the convergence of the sample distribution in the \gls{tv} distance.
We choose \( N = 50 \) \MoreauName{} parameters with endpoints \( t_1 = \num{1e-4} \) and \( t_N = \num{1e-2} \), each of which is used \( K = 20 \) times in the inner loop in \Cref{algo}.
The step  size for \gls{daz} is chosen as $\tau_n=\frac{t_n}{2}$ based on \eqref{eq:step_size}.
To showcase the influence of the initialization, the experiment is done once with $\pi_0 = \Nc(0,1)$ and once with $\pi_0 = \delta_0$.
The samples obtained from the Markov chains are then compared to the discretized target distribution in TV distance.
To facilitate the interpretation of the results and to remove the influence of the finite sample and the discretization of the target distribution, we also plot the \gls{tv} distance of a ground truth sample of the same finite size. (Note that the Gaussian mixture allows for easy direct sampling.)

The results in \Cref{fig:mixture_potentials} demonstrate that the \gls{daz} sample converges to the reference sample significantly faster than the other methods, which is due to the the convexification of the \MoreauName{} envelope combined with the possibility of using larger steps.
Indeed, for this experiment our method showed the greatest improvement over the others, which we hypothesize is due to the nonconvex potential.
Combining DAZ with \gls{rock} leads to an additional speed increase at the beginning of the iterations. Regarding the slow convergence of \gls{ula} and \gls{myula} we want to point out that for these methods which use only the smallest step size the trade-off between bias and convergence speed is especially crucial. While it might be possible to increase the convergence speed by using larger step sizes at the cost of an increased bias, we decided to stick to the chosen step-sizes as \gls{rock}---a comparison method that also uses only the smallest Moreau parameter---converges rather quickly.
The remaining error of the reference sample (denoted as GT in \Cref{fig:mixture_potentials}) of approximately \num{0.4} in \gls{tv} to the reference distribution is a consequence of the finite number of chains and the computation of the \gls{tv} distance rather than the sampling method.

\begin{figure}
    \input{tikz_code/TV_error_gmm_rand_init.tikz}
    \hfill
    \input{tikz_code/TV_error_gmm_zero_init.tikz}
    \caption{%
        \Gls{tv} distance between the sample distribution and the target Gaussian mixture.
        Left: Initializing the chains with a standard normal distribution.
        Right: Initializing with a Dirac distribution concentrated at zero. GT denotes the \gls{tv} error of a random ground truth sample of the same size as the number of simulated chains. \gls{daz} converges fastest with additional acceleration by combining it with \gls{rock}. 
    }
\label{fig:mixture_potentials}
\end{figure}

\subsection{TV prior sampling}\label{sec:tv_marginal}
In this experiment, we sample from a multi-dimensional ($d=10$) distribution with potential $U = G \colon \R^d \to \R$ where $G(x) = \sum_{i=1}^{d-1} |x_{i+1} - x_i|$, which is the standard \gls{tv} functional.
The distribution of this potential is not integrable over \( \R^d \), but it is integrable over the quotient space \( \R^d / \{ t \cdot 1 \in \R^d \mid t \in \R \} \).
To account for this in our implementation, we project onto this linear subspace in each iteration by subtracting the mean of the iterate.\footnote{%
    This experiment is not covered by the theory presented in the paper for two reasons.
    First, our theoretical results do not cover projections.
    However, we believe that this particular projection is unproblematic since the subgradient of TV and, consequently, its proximal map have zero mean when the input has zero mean.
    In addition, the Gaussian random vector added in each iteration has zero mean in expectation.
    Thus, the projection will be close to the identity in practice.
    Second, the potential is not superexponential.
    Both of these issues can be avoided by compensating for the non-trivial kernel of TV via considering a regularized version of the potential $U + \frac \delta 2 \|\cdot\|^2, 0<\delta \ll1$ as is done in the two subsequent experiments. We decided to stick to the projection instead in this experiment in order not to modify the potential and since the solution of adding a squared term is covered in the subsequent experiments anyway.
}

We initialize all chains with samples from $\Nc(0,0.1\cdot I), I\in \R^{d \times d}$ and simulate them for \( N =\num{1000} \) Moreau levels with endpoints $t_1 = \num{2e-4}$ and $t_N = \num{1e-1}$, each of which is used \( K = 1 \) times.
In accordance to~\eqref{eq:step_size} we set the step size to $\tau_n = \frac {t_n} {2}$.
The proximal map of $G$ is computed with the efficient method proposed in~\cite{pock2016total} that is based on dynamic programming.

The high dimensionality of this problem prohibits the computation of \( \pi \) and any distance to it.
Moreover, since the density is only well-defined on a subspace of $\R^d$ also the computation of distances to marginals of the density is not straight forward (contrary to the subsequent TV experiments).
However, by the transformation theorem for integrals, if $X\in \R^d / \{ t \cdot 1 \in \R^d \mid t \in \R \} $ follows the distribution $\frac{1}{Z}\exp(-U(x))$, then the finite differences $X_{i+1}-X_i$ for $i=1, \dotsc, d - 1$ follow a Laplace distribution.
Thus, we can compare the distribution of the finite differences of the samples of the various sampling algorithms at any iteration \( k \), that we denote with \( \Pi_{(i+1) \to i}(\pi_k) \)\footnote{%
    We used \texttt{numpy.histogram} with options \texttt{bins="auto"} and \texttt{density=True} to generate histograms based on the samples (\cf~\url{https://numpy.org/doc/stable/reference/generated/numpy.histogram.html}).
} for $i=1, \dotsc, d - 1$ to the known reference distribution of these finite differences, that we denote with  \( \Pi(\pi) \propto \exp(-|\,\cdot\,|) \). (The reference distribution of the finite differences is independent of the index.)
This evaluation gives us access to \num{9} marginal distributions that we could compare.
For plotting we pick only three representative marginals, namely those with the smallest, median, and largest \gls{tv} error at the end of the iterations\footnote{As the empirical \gls{tv} error oscillates significantly, determining the three representative marginals using only the \gls{tv} error at the final iteration might lead to skewed results. Thus, we compute for each marginal the average \gls{tv} error across the last \num{50} iterations. The obtained averages are used to find the three representative marginals.}.
The results are shown in~\Cref{fig:tv_prior_tv_distance}, where we find that \gls{daz} and \gls{ald} converge significantly faster than \gls{ula} and \gls{myula}.
\gls{rock} gives mediocre performance while the combination of \gls{daz} and \gls{rock} works very well.
We hypothesize that \gls{ald} being on par with our method can be attributed to the convexity of $U$.
The plots in~\Cref{fig:tv_prior_marginals} visualize the negative-log histograms of the three representative marginal distribution obtained from the final steps of the simulated chains (\ie, $-\log(\Pi_{(i+1)\to i}(\pi_{1000}))$ for the three representative \( i\)'s) for the different algorithms.
Again, we find that with \gls{daz}, \gls{ald}, and \gls{daz}-\gls{rock} the ground truth potential is approximated accurately.




\begin{figure}
    \input{tikz_code/tv_prior_tv_distance_05}
    \hfill
    \input{tikz_code/tv_prior_tv_distance_095}
    \hfill
    \input{tikz_code/tv_prior_tv_distance_005}
    \caption{\Gls{tv} prior sampling. \Gls{tv} distance between three finite difference marginals and the known ground truth. \gls{daz} and \gls{ald} converge fastest, again, with additional acceleration by adding \gls{rock} in the inner loop of \gls{daz}. The good performance of \gls{ald} might be due to convexity of the potential.}
\label{fig:tv_prior_tv_distance}    
\end{figure}

\begin{figure}
    \input{tikz_code/tv_prior_marginal_05}
    \hfill
    \input{tikz_code/tv_prior_marginal_095}
    \hfill
    \input{tikz_code/tv_prior_marginal_005}
    \caption{\Gls{tv} difference marginals. The closer a method resembles the absolute value function, the better. We find that the samples obtained with \gls{daz}, \gls{ald}, and \gls{daz}-\gls{rock} approximate the target significantly more accurately.}
\label{fig:tv_prior_marginals} 
\end{figure}


\subsection{TV-L2 denoising on a chain}\label{sec:TV_chain}
The next example we consider is \gls{tv} denoising on a chain, where $F, G : \R^d \rightarrow \R$, $F(x) \coloneqq \frac{1}{2\sigma^2}\|x-y\|^2$ and $G(x) = \lambda\sum_{i=1}^{d-1} |x_{i+1}-x_i|$ with $d = 100$, $\sigma=0.1$, and $\lambda=30$.
To generate the data $y\in\R^d$ we first construct a piecewise constant vector $y^\dagger\in\R^d$ with values 
\begin{equation}
    y^\dagger_i= \begin{cases}
        -3\quad&\text{for } i=1,\dots,10\\
-1\quad&\text{for } i=11,\dots,30\\
3\quad&\text{for } i=31,\dots,35\\
2\quad&\text{for } i=36,\dots,75\\
0\quad&\text{else}.
    \end{cases}
\end{equation}
and compute $y = y^\dagger+\sigma z$ with $z\sim\Nc(0,I)$ and $I\in\R^{d\times d}$ the identity matrix.
The potential satisfies \Cref{ass} and is superexponential and, consequently, the theory presented in this paper applies.
The endpoints of the \MoreauName{} parameters are \( t_N = \num{1e-3} \) and \( t_1 = \num{1e-4} \), each of which is used \( K = 20 \) times, and again $\tau_n=\frac{t_n}{2}$.
We compute the proximal map of $\lambda G$ with the efficient method proposed in \cite{pock2016total} which is based on dynamic programming.

Like in the previous section, the high dimensionality makes the computation of the reference distribution and any distances to it prohibitively expensive.
However, estimates of the \( d \) reference marginal distributions---that we denote with \( \Pi_1(\pi), \dotsc, \Pi_d(\pi) \)---can be computed efficiently using \gls{bp} algorithms \cite{tapfre03,knobelreiter2020belief,narnhofer2022posterior}.
In turn, this enables the verification of the convergence to the reference distribution through the computation of the \gls{tv} distances between the (one-dimensional) marginals obtained from \gls{bp} and those obtained from the various sampling algorithms.
Similar to the previous section, we choose three representative marginals to visualize by estimating the averaging the \gls{tv} distance over the last \num{20} iterations and then choose those marginals that correspond to the \num{5}th, \num{50}th, and \num{95}th percentile.
The results are shown in \Cref{fig:tv_denoising_chain}.
We find that \gls{rock} and \gls{daz}-\gls{rock} perform best, closely followed by \gls{daz} and afterwards \gls{ald}. As mentioned above, we believe that the advantages of \gls{daz} are most prevalent in the nonconvex setting, so that in this experiment the acceleration by \gls{rock} is already at a similar level.
The double dip in the convergence (especially visible for \gls{ula} and \gls{myula}) might be a consequence of the fact that we measure the distance to $\pi$, whereas the chains target a biased version thereof.


In \cref{fig:typical_set} we plot for each method for one chain the value of the potential $(U(X_k))_k$. For convex potentials this value concentrates on a so-called \emph{typical set} close to $\E[U(X)]$\footnote{up to the discrepancy induced by the difference between $U$ and $U^t$} \cite{pereyra2017maximum,pereyra2020skrock} so that we can use the convergence speed of the sequence $(U(X_k))_k$ as a proxy for the convergence speed of the Markov chain itself. We find in this experiment that \gls{daz} and \gls{ald} converge fastest, closely followed by \gls{rock}. \Gls{daz}-\gls{rock} converges roughly at the same speed, however, including strong oscillations. \Gls{ula} and \gls{myula} converge significantly slower.

\subsection{TV-L2 denoising for images}\label{sec:TV_image}
For image denoising, $F, G : \R^{N \times M} \rightarrow \R$ with $F(x) \coloneqq \frac{1}{2\sigma^2}\|x - y\|^2$ again and $G(x) = \lambda\sum_{i,j} |(Dx)_{i,j,1}|+|(Dx)_{i,j,2}|$ is the anisotropic \gls{tv}, where $D : \R^{N \times M} \rightarrow \R^{N \times M \times 2}$ is a forward-finite-differences operator \cite[Section 6.1]{chambolle2011first}.
We set $N=M=200$, $\sigma=0.05$, and $\lambda=30$. The data $y$ was again computed as $y=y^\dagger + \sigma z$ with $z\in \R^{N\times M}$, $(z_{i,j})_{i,j}$ i.i.d, $z_{i,j}\sim\Nc(0,1)$ and $y^\dagger$ a $200\times 200$ crop of the ground truth watercastle image.
For the computation of the proximal map of $\lambda G$ we use again \cite{pock2016total}.
The endpoints of the \MoreauName{} parameters are \( t_N = \num{1e-3} \) and \( t_1 = \num{1e-5} \), each of which is used \( K = 20 \) times, and again $\tau_n = \frac{t_n}{2}$.

As in \Cref{sec:TV_chain}, we obtain \( M \times N \) reference marginal distributions using the \gls{bp} algorithm and compute the one-dimensional \gls{tv} distances to the sample distributions for each of them.
Again, as in \Cref{sec:TV_chain}, we plot in \Cref{fig:tv_denoising_image} the convergence of three representative marginals,
and in \cref{fig:typical_set} on the right the convergence to the typical set. We observe best convergence from \gls{daz}, \gls{ald}, and \gls{rock}. \Gls{daz} converges fastest on the medium marginal, on par with \gls{ald} on the slowest marginal, and \gls{rock} being fastest on the fastest converging marginal. In this experiment the combination \gls{daz}-\gls{rock} induces strong oscillations which, however, do not disrupt the overall convergence. We speculate that the combination of \gls{rock} with annealing, \ie, successive alteration of the target distribution, exhibits slightly instable behavior. 
The convergence to the typical set is in this experiment fastest for \gls{daz}, \gls{ald}, and \gls{rock}\footnote{Note that we focus here on the time until stationarity is reached, not reached the value of $U(x)$.}. The fact that for denoising on a chain the value of $U(x)$ increases during iterations, for images, however, it decreases is caused since for chains we initialized at zero and for images at the given noisy observation.

\begin{figure}
    \input{tikz_code/TV_error_TV_chain_05.tikz}
    \hfill
    \input{tikz_code/TV_error_TV_chain_095.tikz}
    \hfill
    \input{tikz_code/TV_error_TV_chain_005.tikz}
    \caption{\Gls{tv} denoising on a chain. \Gls{tv} distance between three different marginals of sample distribution and the target for TV-denoising. Fastest convergence by \gls{daz}, \gls{rock}, and \gls{daz}-\gls{rock}. The double dip behavior might be a consequence of the fact that we compare to the target distribution and not to the stationary measure of each chain.}
\label{fig:tv_denoising_chain}    
\end{figure}


\begin{figure}
    \input{tikz_code/TV_error_TV_image_05.tikz}
    \hfill
    \input{tikz_code/TV_error_TV_image_095.tikz}
    \hfill
    \input{tikz_code/TV_error_TV_image_005.tikz}
    \caption{\Gls{tv} denoising on an image. \Gls{tv} distance between three different marginals of sample distribution and the target for TV-denoising. Fastest convergence for \gls{daz}, \gls{ald}, and \gls{rock}. \gls{daz}-\gls{rock} exhibits strong oscillations which, however, do not bother the overall convergence. The piecewise constant shape of \gls{rock} is due to the fact that for \gls{rock} we count one iteration per prox evaluation which is plotted by piecewise constant upsampling.}
\label{fig:tv_denoising_image}    
\end{figure}


\begin{figure}
    \input{tikz_code/typical_set_TV_chain.tikz}
    \hfill
    \input{tikz_code/typical_set_TV_image.tikz}
    \caption{TV denoising convergence to the typical set. Left: denoising on a chain, right: denoising on images. Fastest convergence for \gls{daz}, \gls{ald}, and \gls{rock}. \gls{daz}-\gls{rock} exhibits strong oscillations which, however, do not bother the overall convergence. The piecewise constant shape of \gls{rock} is due to the fact that for \gls{rock} we count one iteration per prox evaluation which is plotted by piecewise constant upsampling.
    }
\label{fig:typical_set}
\end{figure}

\subsection{Sampling from a deep prior}%
\label{sec:TDV}
To showcase the applicability of the proposed sampling algorithm to nonconvex potentials in high-dimensional settings, we consider the \gls{tdv} potential introduced in~\cite{kobler2020total}. This potential is twice continuously differentiable and the second derivative is bounded and, consequently, the potential is Lipschitz continuous and weakly convex\footnote{Unfortunately, convexity outside a ball is difficult to verify for TDV and similar deep models.}. However, a naive usage of the potential does not meet the integrability requirements and instead we consider the potential $U(x) = \lambda \operatorname{TDV}(x) + \frac{1}{2 \gamma_1} \bigl(\tfrac{1}{d}\sum_{i=1}^d x_i - \mu\bigr)^2 + \frac{1}{2 \gamma_2}(\operatorname{var}[x] - \sigma^2)^2$ which quadratically penalizes the statistical deviations of a given patch to ensure integrability. In addition to the potential now being integrable we can also enforce the mean and variance of the samples staying close to $\mu$ and $\sigma^2$ which are chosen as the empirical mean and variance of the training dataset which were computed over \num{10000} patches of size $(96 \times 96)$ extracted from the BSDS~\cite{bsds} dataset.
For our experiment, we utilize the publicly available \gls{tdv} weights\footnote{See: \url{https://github.com/VLOGroup/tdv}} that were obtained from optimizing the denoising performance on perturbed BSDS500~\cite{bsds} grayscale image patches of size $96 \times 96$.

Throughout this experiment, we use $\lambda = \num{8}$, $\gamma_1 = \gamma_2 = \num{1e-8}$, and work on $d=256\times 256$-dimensional image patches.
We simulate $N=\num{5}$ \MoreauName{} parameters whose endpoints are are chosen as $t_1 = \num{1e-4}$ to $t_N = 5 t_1$ and each of which is used for $K=\num{200}$ Langevin steps, which results in a total of \num{1000} iterations.
The Langevin step size in each level is set to to $\tau_n = t_n /2$ in accordance with \Cref{rmk:step_size}.
The proximal operator of $G$ is computed with \gls{apgd}, which is given in \Cref{algo:apgd}, where we perform gradient steps on the potential and proximal steps on the proximity term.
The proximity term could also be handled through gradient steps, but we observed slightly faster convergence when splitting the terms.
All trajectories are initialized at the same $x_0 = \mu + \sigma z, \, z \sim \mathcal{N}(0, I), I \in \R^{d \times d}$.

The high dimensionality and the nonconvexity of the potential prohibits a quantitative analysis of the convergence.
Consequently, we only provide a qualitative comparison of the methods in \Cref{fig:tdv results}.
The results show that the chains of \gls{ald} and \gls{daz} reach a reasonable sample after around \num{200} iterations, whereas \gls{myula} and \gls{ula} require roughly \num{600} to \num{800} iterations to reach a result that is of similar quality.

In addition, we point out that we sometimes observed instabilities for \gls{ald} when the step sizes were too large.
This underpins the relevance of the proposed sampling algorithm from a stability and convergence speed perspective and supports the findings regarding possible step sizes in \cref{rmk:step_size}.

\begin{figure}
    \centering
	\def\imwidth{2.5cm}
	\def\ppad{.25cm}
	\begin{tikzpicture}
		\foreach [count=\imethod] \method/\anno in {ula/ULA, ald/ALD, myula/MYULA, daz/DAZ} {
			\pgfmathsetlengthmacro{\yyanno}{-(\imethod) * (1*\imwidth + \ppad)}
			\node [rotate=90] at (1.0cm, \yyanno) {\anno};
            \draw [gray, thick] (1.2cm, \yyanno + 1cm) -- ++(0,-2cm);
			\foreach [count=\iidx] \idx/\idxreadable in {00000/ Initialization, 00200/Iteration 200, 00400/400, 00600/600, 00800/800} {
				\pgfmathsetlengthmacro{\xx}{\iidx * (\imwidth + \ppad)}
				\ifthenelse{\imethod=1}{\node at (\xx, -1.25cm) {\idxreadable};}{}
				\pgfmathsetlengthmacro{\yymean}{-\imethod * (\imwidth + \ppad)}
				\node at (\xx, \yymean) {\includegraphics[width=\imwidth]{./images/TDV/\method/0-\idx.png}};
			}
		}
	\end{tikzpicture}
    \caption{%
        Exemplary trajectories obtained by various sampling algorithms applied to the \gls{tdv} potential.
        From top to bottom: \gls{ula}; \gls{ald}; \gls{myula}; \gls{daz};. From left to right: Iterations 0; 200; 400; 600; 800. \gls{ula} and \gls{myula} take significantly longer to reach a characteristic sample visually.
    }%
    \label{fig:tdv results}
\end{figure}

\subsection{Accelerated MRI reconstruction}\label{sec:MRI}
In this section we showcase the applicability of the proposed sampling algorithm to high-dimensional inverse problems with a nontrivial forward operator inspired by accelerated MRI.
We consider the problem of recovering a real-valued image from undersampled and noisy Fourier data and thus \( F, G : \R^{M\times N} \to \R \) with \( F(x) \coloneqq \tfrac{\lambda}{2} \| M\Fourier x - y \|_2^2 \), where \( \Fourier : \R^{M\times N} \to \mathbb{C}^{M \times (\lfloor N/2 \rfloor + 1)} \) is the two-dimensional Fourier transform that accounts for the conjugate symmetry of the spectrum of a real signal and \( M : \mathbb{C}^{ M \times (\lfloor N/2 \rfloor + 1)} \to \mathbb{C}^{n} \) is a frequency selection operator that models the undersampling that is used to accelerate imaging speed.
The function \( G = \operatorname{EBM} + \tfrac{1}{2\gamma} \|\,\cdot\,\|^2 \) incorporates the learned energy-based model \( \operatorname{EBM} \) from \cite{zach2023stable}.
This model is a cascade of convolution operators with stride two and point-wise leaky-rectified-linear activation functions, followed by a linear layer that maps to a scalar and finally the absolute value to ensure boundedness from below.
To ensure weak convexity of \( G \),\footnote{Again, convexity outside a ball is difficult to verify for deep models.} we replaced the leaky-rectified-linear activation functions \( x \mapsto \max(\alpha x, x) \) with \( \alpha = \num{0.05}\) with the surrogate \( x \mapsto \tfrac{1}{\beta}\log\bigl(\exp(\alpha \beta x) + \exp(\beta x)\bigr) \) with \( \beta = \num{1000} \).
This surrogate is infinitely often differentiable with bounded second derivative and, consequently, the potential is Lipschitz continuous and weakly convex.
In addition, the small quadratic norm with \( \gamma = \num{1e9} \) ensures the integrability of the Gibbs distribution.
The data \( y \in \mathbb{C}^n \) are constructed by \( y = M\Fourier x^\ast + \sigma z \) where \(  x^\ast \in \R^{320 \times 320} \) is the reference root-sum-of-squares reconstruction of the 17th slice in the file \texttt{file1000005.h5} in the \texttt{multicoil\_train} folder of the fastMRI knee dataset \cite{Knoll2020} and \( z \sim \mathcal{N}(0, I) \).
Like in the simulation studies in the original publication \cite{zach2023stable}, the intensity values of the reference root-sum-of-squares reconstruction were affinely mapped such that their minimum is zero and their maximum is one, such that the intensity values consistent with those in the training.
The standard deviation of the noise was chosen as \( \sigma = \num{2e-2} \) and we set \( \lambda = \num{1e5} \) based on manual search optimized on visually appealing reconstructions.

For the sampling algorithms, \( \tau_1 = \num{7.5e-3} \) is determined by the step size used in the \gls{ula} algorithm used in training the model.
We chose \( N = 5\), \( K = 200 \) and the endpoints or the sequence of Moreau parameters as \( t_1 = \frac{\tau_1}{\tau_1 - \tau\lambda} \) (to comply with step size restrictions) and \( t_N = 5t_1 \).
As in the previous section, we computed the proximal map of \( G \) using \gls{apgd} where we took gradient steps on \( G \) and proximal steps on the proximity term.
We ran \num{500} parallel chains all initialized at the naive reconstruction \( \Fourier^\ast M^\ast y \) that is shown in~\Cref{fig:mri data} along with the reference reconstruction and a visualization of the data.
\begin{figure}
    \centering
    \def\imwidth{3cm}
	\def\ppad{.5cm}
    \begin{tikzpicture}
    \foreach [count=\iim] \im/\anno in {reference-rss/x, data/ M^\ast\log(|y|), naive/\Fourier^\ast M^\ast y} {
        \pgfmathsetlengthmacro{\xx}{\iim * (\imwidth + \ppad)}
        \node at (\xx, 0) {\includegraphics[height=\imwidth]{./images/mri/\im.png}};
        \node at (\xx, 1.8) {\( \anno \)};
        \node at (\xx, 0 - 1.7cm) {\csvreader[no head]{./images/mri/\im_vals.txt}{}{\csvcoli\drawcolorbar\ \csvcolii}};
    }
    \end{tikzpicture}
    \caption{%
        From left to right:
        reference reconstruction; visualization of the data; least-squares reconstruction.%
    }%
    \label{fig:mri data}
\end{figure}

Like in the previous section, the high dimensionality and nonconvexity makes this an extremely challenging problem.
Due to the absence of any ground-truth samples we describe the results of the four sampling algorithms shown in \Cref{fig:mri results} again only qualitatively.
The marginal standard deviations obtained by the gradient-based samplers \gls{ula} and \gls{ald} are significantly more blurry than those obtained by the proximal-based samplers \gls{myula} and \gls{daz}.
In addition, \gls{ald} and \gls{daz} converge significantly faster than \gls{myula} and \gls{ula}.
This is examplified by the backfolding artifact indicated by the gray arrow in \Cref{fig:mri results}, that is clearly visible in the \gls{mmse} estimate even after \num{1000} iterations of \gls{myula} and \gls{ula}.
In contrast, it is almost fully removed after \num{600} iterations of \gls{daz}.
\begin{figure}
    \centering
	\def\imwidth{2.05cm}
	\def\ppad{.5cm}
	\begin{tikzpicture}
		\foreach [count=\imethod] \method/\anno in {ula/ULA, ald/ALD, myula/MYULA, daz/DAZ} {
			\pgfmathsetlengthmacro{\yyanno}{-(\imethod + .22) * (2*\imwidth + \ppad)}
			\node [rotate=90] at (1.1cm, \yyanno) {\anno};
            \draw [gray, thick] (1.35cm, \yyanno + 1.5cm) -- ++(0,-3cm);
			\foreach [count=\iidx] \idx/\idxreadable in {00200/Iteration 200, 00400/400, 00600/600, 00800/800, 01000/1000} {
				\pgfmathsetlengthmacro{\xx}{\iidx * (\imwidth + \ppad)}
				\ifthenelse{\imethod=1}{\node at (\xx, -3.3cm) {\idxreadable};}{}
				\pgfmathsetlengthmacro{\yymean}{-\imethod * (2*\imwidth + \ppad)}
				\pgfmathsetlengthmacro{\yystd}{-\imethod * (2*\imwidth + \ppad) - \imwidth}
				\node at (\xx, \yymean) {\includegraphics[width=\imwidth]{./images/mri/\method/mean/\idx.png}};
				\node at (\xx, \yystd) {\includegraphics[width=\imwidth]{./images/mri/\method/std/\idx.png}};
                \node at (\xx, \yystd - 1.2cm) {\csvreader[no head]{./images/mri/\method/std/\idx_vals.txt}{}{\fontsize{6pt}{12}\selectfont\pgfmathparse{\csvcoli*100}\pgfmathprintnumber{\pgfmathresult}\drawcolorbar\ \pgfmathparse{\csvcolii*100}\pgfmathprintnumber{\pgfmathresult}}};
			}
            \draw[-{Triangle[width=9pt,length=5pt]}, line width=5pt, draw=gray](8.95, \yyanno+2.4cm) -- ++(-.5, -.5);
		}
	\end{tikzpicture}
    \caption{%
        \Gls{mmse} and standard deviation estimates obtained by various sampling algorithms.
        From top to bottom: \gls{ula}; \gls{ald}; \gls{myula}; \gls{daz}.
        For each algorithm, the \gls{mmse} estimate is shown on top of the variance estimate.
        From left to right: Iterations 200; 400; 600; 800; 1000.
        The colormap values are multiplied by \num{100} and apply only to the standard deviation.
        The gray arrow indicates a prominent backfolding artifact. \gls{ula} and \gls{ald} lead to blurrier results. \gls{ald} and \gls{daz} remove the artifact significantly faster.
    }%
    \label{fig:mri results}
\end{figure}
\section{Conclusion}
In this article, we proposed a method for the efficient sampling from Gibbs distributions $\pi$ whose potential may be nondifferentiable and nonconvex.
Inspired by diffusion models, we consider a sequence of distributions $\pi^t$ that has favorable properties for sampling when \( t \) is large and approaches the target \( \pi \) as $t\rightarrow 0$.
The sequence is obtained by replacing the parts of the potential with its \MoreauName{} envelope.
Within our approach we then successively sample from $\pi^t$ for decreasing values of $t$ by constructing a Markov chain from the corresponding Langevin diffusion.
We proved ergodicity of the chain for sampling from $\pi^t$ for fixed $t$ and convergence to the target density in \gls{tv} when $t$ is sufficiently small.
Moreover, we showed that the map $t \mapsto \pi^t$ is Lipschitz continuous in the \gls{tv} norm, which justifies the approach of successively sampling from $\pi^t$ for decreasing $t$.
In addition, we proved that all distributions $\pi^t$ can be understood as a zero-temperature limit of the Gibbs distributions corresponding to a variance exploding diffusion model for $\pi$.

An extensive set of numerical experiments that contains one-dimensional toy problems and high-dimensional Bayesian inverse problems in imaging confirmed the efficacy of the method compared to \gls{ula}, \gls{myula}, \gls{ald}, and \gls{rock}. 
In addition to the broad applicability of \gls{daz} to potentials of rather general form, the conducted experiments confirm that the proposed method yields a significant speedup particularly for nonconvex or not strongly convex potentials (\cref{sec:GM,sec:tv_marginal}). In the strongly convex case, \gls{daz} yields faster convergence than \gls{ula} or \gls{myula} due to the relaxed step size conditions and provides performance comparable to \gls{ald} and \gls{rock}.

\paragraph{Limitations and future work}
Within the sampling procedure, we successively sample from the distributions $\pi^t$ by initializing each Markov chain with the last iterate of the previous chain.
Contrary, \eg, in diffusion models, it is possible to discretize the backward SDE in order to obtain a theoretically well-founded time step from $\pi^{t+\Delta t}\rightarrow\pi^{t}$ where $\Delta t$ denotes the size of the time step.
Future work will investigate if, \eg, the sequence $(\pi^t)_t$ admits a governing transport equation that would allow a similar discretization of the time evolution $t \mapsto \pi^t$.

An additional direction for future work is to extend the proposed scheme to Bregman-Moreau envelopes similar to \cite{lau2022bregman} and the theoretical investigation of advanced sampling schemes like \gls{rock} as inner algorithms for \gls{daz}.

\bibliographystyle{siamplain}
\bibliography{references}

\appendix
\section{Proofs}
\subsection{Ergodicity of \texorpdfstring{\gls{daz}}{DAZ}}\label{sec:appendix_erg}
\subsubsection{Proof of \texorpdfstring{\Cref{lemma:prox_time_cont}}{Lemma 4.6}}\label{sec:prox_time_cont}
\begin{proof}
    In the following, we will frequently denote an element $p\in {\prox}_{\tau G}(x)$ as $p = p_\tau(x)$  for simplicity. If it is clear from context, we will also omit either the argument $x$ or the parameter $\tau$ and denote the prox as $p_\tau$, or simply $p$. We begin with the first assertion by letting $(x_n,t_n)\rightarrow (x,t)$ in $\R^d\times(0,\frac{1}{\ModulusWeakConvexity})$. By definition of the proximal map, we have that
    \[
        \frac{1}{2t_n}\|x_n-p_{t_n}(x_n)\|^2+G(p_{t_n}(x_n))\leq \frac{1}{2t_n}\|x_n\|^2+G(0).
    \]
    The right-hand side is bounded by convergence of $x_n\rightarrow x$, $t_n\rightarrow t$ and, thus, by boundedness from below of $G$, it follows that $(p_{t_n}(x_n))_n$ is bounded. Therefore, there exists a convergent subsequence $p_{t_n}(x_n)\rightarrow \hat{p}$, which we shall not relabel for the sake of a simpler notation. It follows for arbitrary $y\in\R^d$ that
    \begin{equation}
        \begin{aligned}
            \frac{1}{2t}\|x-\hat{p}\|^2+G(\hat{p}) &\underset{(i)}{=} \lim_{n\rightarrow\infty}\frac{1}{2t_n}\|x_n-p_{t_n}(x_n)\|^2+G(p_{t_n}(x_n)) \\
            &\underset{(ii)}{\leq} \lim_{n\rightarrow\infty}\frac{1}{2t_n}\|x_n-y\|^2+G(y)\\
            &= \frac{1}{2t}\|x-y\|^2+G(y)
        \end{aligned}
    \end{equation}
    where $(i)$ is a consequence of continuity and $(ii)$ follows from optimality of $p_{t_n}(x_n)$.
    Since $y$ was arbitrary we can conclude $\hat{p} \in \prox_{tG}(x)$, which is single-valued for $t<\frac{1}{\ModulusWeakConvexity}$.
    Convergence of the original sequence $(p_{t_n}(x_n))_n$ follows from a standard subsequence argument.
    
    Next, we consider the Lipschitz continuity for fixed $t$.
    The optimality condition of the proximal map implied that for any $z\in\R^d$ there exists a $v_z\in\partial G(p(z))$ such that $0 = \frac{1}{t}(p(z) - z) + v_z$. It follows
    \begin{equation}\label{eq:moreau_lipschitz}
        \begin{aligned}
            \|p(x)-p(y)\|^2 &= \langle p(x)-p(y), x-tv_x-(y-tv_y)\rangle\\
            &= \langle p(x)-p(y), x-y\rangle - t\langle p(x)-p(y), v_x-v_y\rangle.
        \end{aligned}
    \end{equation}
    As mentioned in \Cref{remark:ass}, the regular subgradient of the convex function $G + \frac{\ModulusWeakConvexity}{2}\|\cdot\|^2$ satisfies $\partial(G + \frac{\ModulusWeakConvexity}{2}\|\cdot\|^2)(x) = \partial G(x) + \ModulusWeakConvexity x$ (see \cite[10.10 Exercise]{rockafellar2009variational}). Therefore, it holds that
    \[
        0\leq \langle p(x)-p(y), (v_x + \ModulusWeakConvexity p(x)) - (v_y + \ModulusWeakConvexity p(y))\rangle = \langle p(x)-p(y), v_x - v_y\rangle + \ModulusWeakConvexity\|p(x)-p(y)\|^2.
    \]
    Plugging this into \eqref{eq:moreau_lipschitz} yields that
    \[
        \|p(x)-p(y)\|^2 \leq \langle p(x)-p(y), x-y\rangle + t\ModulusWeakConvexity\|p(x)-p(y)\|^2
    \]
    Thus, $\|p(x)-p(y)\|^2 \leq \frac{1}{1-\ModulusWeakConvexity t}\langle p(x)-p(y), x-y\rangle$.
    Applying the Cauchy-Schwartz inequality yields the desired result.
\end{proof}
\subsubsection{Proof of \texorpdfstring{\Cref{lemma:Moreau_diffble_small_t}}{Lemma 4.7}}\label{sec:Moreau_diffble_small_t}
\begin{proof}
    Let $v\in \R^d$ be arbitrary.
    By the definition of the proximal map, we find that
    \begin{equation}
        \begin{aligned}
            M^t_G(x+v)-M^t_G(x)&\leq \frac{1}{2t}\|x+v-p_t(x)\|^2 + G(p_t(x)) - \frac{1}{2t}\|x-p_t(x)\|^2 - G(p_t(x))\\
            &=\frac{1}{t}\langle v, x-p_t(x)\rangle + \frac{1}{2t} \|v\|^2.
        \end{aligned}
    \end{equation}
    Similarly, it holds true that $M^t_G(x+v)-M^t_G(x)\geq\frac{1}{t}\langle v, x-p_t(x+v)\rangle + \frac{1}{2t} \|v\|^2$. In summary, we have that
    \begin{equation}
        \begin{aligned}
            |M^t_G(x+v)-M^t_G(x) - \langle v, \tfrac{1}{t}(x-p_t(x))\rangle| & \leq \tfrac{1}{t} \|v\|\|p_t(x+v)-p_t(x)\| + \tfrac{1}{2t} \|v\|^2 \\
            &\leq \tfrac{1}{t}\|v\|(\tfrac{1}{1-\ModulusWeakConvexity t}\|v\| + \tfrac{1}{2}\|v\|) = o(\|v\|).
        \end{aligned}
    \end{equation}
\end{proof}
\subsubsection{Proof of \texorpdfstring{\Cref{lemma:prox_coercive}}{Lemma 4.9}}\label{sec:appendix_prox_coervice}
\begin{proof}
    By continuity, the set $\prox_{tG}(x)$ is closed and, consequently, the infimum is a minimum.
    Denote $\hat{p}_t(x) = \arg\min_{p\in \prox_{tG}(x)}\|p\|$.
    The optimality condition
    \[
        0\in \tfrac{1}{t}(\hat{p}_t(x)-x) + \partial G(\hat{p}_t(x))
    \]
    implies the existence of a $v_t(x)\in\partial G(\hat{p}_t(x))$ such that $\|x\|\leq \|\hat{p}_t(x)\| + t\|v_t(x)\|$.
    Now assume to the contrary, there exist sequences $R_n\rightarrow\infty$, $\|x_n\|\geq R_n$ with $\|\hat{p}_t(x_n)\|\leq C<\infty$ for all $n\in\N$. But then, the local boundedness of $\partial G$ (see \Cref{remark:ass}) implies a uniform bound on $\|v_t(x_n)\|$ and thus on $\|x_n\|$, which is a contradiction.
\end{proof}
\subsubsection{Proof of \texorpdfstring{\Cref{Moreau_diffble_outside}}{Corollary 4.10}}\label{sec:appendix_moreau_diffble_out}
\begin{proof}
    By \Cref{lemma:prox_coercive} there exists $R>0$ such that for \( x \in \BallC{R}{0} \), $\prox_{tG}(x) \subset \BallC{M}{0}$ where $\ConvexityRadius$ is the radius outside of which \( G \) is convex, see \Cref{ass:convex outside ball}.
    Thus, the map $p\mapsto\frac{1}{2t}\|x-p\|^2 + G(p)$ is strictly convex outside $\Ball{\ConvexityRadius}{0}$.
    Since all minimizers of this map are elements of $\BallC{\ConvexityRadius}{0}$, the minimizer has to be unique, that is, $\prox_{tG}(x)$ is single-valued.
    As in \Cref{lemma:prox_time_cont}, we obtain Lipschitz continuity of the proximal map, but using now convexity of $G$ outside the ball $\Ball{\ConvexityRadius}{0}$ instead of weak convexity.
    This also implies differentiability of $M^t_G$.
    Regarding the Lipschitz continuity of $\nabla M^t_G$, let us denote $q_z=z-p(z)$ for $z=x,y$, so that $\nabla M^t_G(z) = \frac{1}{t}q_z$.
    Moreover, by the definition of the prox $q_z\in t \partial G(p(z)) = t\partial G(z-q_z)$.
    As a consequence, since $p(z) \not\in \Ball{\ConvexityRadius}{0}$ so that convexity of $G$ outside a ball can be used, $\langle p(x)-p(y), q_x-q_y\rangle=\langle x-q_x - (y-q_y), q_x-q_y\rangle\geq 0$.
    It follows
    \begin{equation}
        \|q_x-q_y\|^2 = \langle x-y, q_x-q_y\rangle - \langle x-q_x - (y-q_y), q_x-q_y\rangle\leq \|x-y\|\|q_x-q_y\|.
    \end{equation}
    Therefore, $\nabla M^t_G$ is Lipschitz continuous with Lipschitz constant $\frac{1}{t}$.
\end{proof}
\subsubsection{Proof of \texorpdfstring{\Cref{cor:moreau_convex_outside}}{Corollary 4.11}}\label{sec:appendix_moreau_convex_outside}
\begin{proof}
    Let $R$ as in the proof of \Cref{Moreau_diffble_outside} be large enough such that $x \in \BallC{R}{0}$ implies that $\prox_{tG}(x) \in \BallC{\ConvexityRadius}{0}$.
    A simple computation then yields for $x,y \in \BallC{R}{0}$
    \begin{equation}
        \begin{aligned}
            M^t_G(\lambda x + (1-\lambda)y)& \leq \frac{1}{2t}\|\lambda p(x) + (1-\lambda)p(y) - (\lambda x + (1-\lambda)y)\|^2  + G(\lambda p(x) + (1-\lambda)p(y))\\
            &\leq \frac{\lambda}{2t}\| p(x) - x\|^2 + \lambda G(p(x))
            + \frac{1-\lambda}{2t}\|p(y) - y\|^2 + (1-\lambda) G(p(y))\\
            &= \lambda M^t_G(x) + (1-\lambda)M^t_G(y)
        \end{aligned}
    \end{equation}
    where the first inequality follows from definition of the \MoreauName{} envelope and the second inequality follows from from convexity of \( \|\,\cdot\,\| \) together with convexity of $G$ outside a ball.
\end{proof}
\subsubsection{Proof of \texorpdfstring{\Cref{lemma:superexp_M}}{Lemma 4.13}}\label{sec:appendix_superexp}
\begin{proof}
    First, we note that $x^*$ is also a minimizer of $M^t_G$.
    Let $R>0$ be such that $M^t_G(x)$ is differentiable at \( x \in \R^d \) that fulfills $\|x\|>R$.
    Pick $M_\rho'>0$ such that $M_\rho'\geq 2M_\rho$ and such that for all \( x \in \R^d \), $\|x-x^*\|\geq M_\rho'$ implies that $\| x \|> R$.
    Let $p=\prox_{tG}(x)$ and where \( x \in \R^d \) is such that  $\|x-x^*\|\geq M_\rho'$.
    We make a case distinction:
    When $\|p-x^*\|\leq \frac{\|x-x^*\|}{2}$, we can compute that
    \begin{equation}
        \begin{aligned}
            \langle \nabla M^t_G(x),x-x^*\rangle &=\langle \tfrac{1}{t}(x-x^*+x^*-p), x-x^*\rangle\\
            &\geq \tfrac{1}{t}(\|x-x^*\|^2-\|p-x^*\|\|x-x^*\|\\
            &\geq \tfrac{1}{2t}\|x-x^*\|^2.
        \end{aligned}
    \end{equation}
    Conversely, assume now $\|p-x^*\|> \frac{\|x-x^*\|}{2}\geq M_\rho$ and let $v\in\partial G(p)$ such that $0=\frac{1}{t}(p-x) + v$. It follows
    \begin{equation}
        \begin{aligned}
            \langle \nabla M^t_G(x),x-x^*\rangle &=\langle  v,x-x^*\rangle\\
            &=\langle v,p-x^*+t v\rangle\\
            &\geq \langle v,p-x^*\rangle \\
            &\underset{(*)}{\geq} \rho\|p-x^*\|^2\geq \frac{\rho}{4}\|x-x^*\|^2
        \end{aligned}
    \end{equation}
    where $(*)$ follows from $G$ being superexponential.
    In summary, we find that for \( x \in \R^d \) such that $\|x-x^*\|\geq M_\rho'$, it holds that
    \begin{equation}\label{eq:superexp_new_rho}
        \langle \nabla M^t_G(x),x-x^*\rangle\geq \min \Bigl(\frac{\rho}{4},\frac{1}{2t}\Bigr)\|x-x^*\|^2.
    \end{equation}
\end{proof}

\subsection{Consistency of \texorpdfstring{\gls{daz}}{DAZ}}%
\label{sec:appendix_consistency}
\subsubsection{Proof of \texorpdfstring{\Cref{lemma:moreau_diffble}}{Lemma 4.18}}\label{sec:moreau_diffble}
\begin{proof}
    Let $0 < s < t < \tmax$.
    By definition of the \MoreauName{} envelope, we find that
    \begin{equation}
        \begin{aligned}
            &\frac{1}{2t}\|x-p_t\|^2+G(p_t) - \frac{1}{2s}\|x-p_t\|^2-G(p_t) \\
            &\qquad\leq M_G^t(x)-M_G^s(x)\\
            &\qquad\leq\frac{1}{2t}\|x-p_s\|^2+G(p_s) - \frac{1}{2s}\|x-p_s\|^2-G(p_s),
        \end{aligned}
    \end{equation}
    and, consequently, that
    \[
        -\frac{1}{2st}\|x-p_t\|^2\leq \frac{M_G^t(x)-M_G^s(x)}{t-s}\leq -\frac{1}{2st}\|x-p_s\|^2.
    \]
    By \Cref{lemma:prox_time_cont} the result follows and, in particular, we obtain $\partial_t M_G^t(x) = -\frac{1}{2t^2}\|x-p_t\|^2$.
\end{proof}
\subsubsection{Proof of \texorpdfstring{\Cref{prop:Moreau_cont}}{Proposition 4.20}}\label{sec:Moreau_cont}
\begin{proof}
    Using \Cref{lemma:moreau_diffble} we find for $0<s<t$
    \begin{equation}
        \begin{aligned}
            |M_G^t(x)-M_G^s(x)| \leq \int_s^t |\partial_\tau M_G^\tau(x)|\;d\tau
            = \frac{1}{2}\int_s^t\frac{1}{\tau^2}\|x-p_\tau\|^2 \;d\tau.
        \end{aligned}
    \end{equation}
    By definition of the \MoreauName{} envelope $\frac{1}{\tau}(x-p_\tau)\in\partial G(p_\tau)$.
    Therefore, $\frac{1}{\tau^2}\|x-p_\tau\|^2\leq \phi(\|x\|)^2$ and consequently $|M_G^t(x)-M_G^s(x)|  \leq \frac{|t-s|}{2}\phi(\|x\|)^2$.
    When $s=0$, it holds that \( |M^0_G(x)-M^t(x)| = \lim_{s\rightarrow 0}|M^s_G(x)-M^t_G(x)| \leq \lim_{s\rightarrow 0}\frac{|t-s|}{2}\phi(\|x\|)^2 = \frac{|0-t|}{2}\phi(\|x\|)^2 \).
\end{proof}

\subsubsection{Miscellaneous results}
\begin{lemma}\label{lemma:U_linear_growth}
    Let $H:\R^d\rightarrow\R$ be locally bounded and convex outside the ball $B_M(0)$ and such that $\int\exp(-H(x))\dd x <\infty$. Then, $H$ grows asymptotically at least linearly, \ie, there exist a $C>0$ such that $H(x)\geq C\|x\|$ for sufficiently large $x$.
\end{lemma}
\begin{proof}
    The proof is similar to \cite[Lemma 4]{durmus2019analysis}.
    Let $M$ be such that $H$ is convex outside of $\Ball{\ConvexityRadius}{0}$. Define $m=\sup_{x\in\Ball{2M}{0}} H(x)<\infty$. We claim now that the set $\{H\leq m+1\}$ is bounded. Assume to the contrary, there exists $y_n$, $\|y_n\|\geq n$ such that $H(y_n)\leq m+1$. By convexity outside a ball it follows that $\text{co}(\Ball{2\ConvexityRadius}{0}\cup \left\{ y_n\right\})\subset \{H\leq m+1\}$ where co denotes the convex hull. To see that, let $z\in \text{co}(\Ball{2\ConvexityRadius}{0}\cup \left\{ y_n\right\})$, \ie, there exists $k\in\N$, $x_1,\dots,x_k\in \Ball{2M}{0}$, and $\lambda_1,\dots,\lambda_{k+1}\geq 0$ with $\sum_i \lambda_i=1$ such that $z=\sum_{i=1}^k \lambda_i x_i + \lambda_{k+1}y_n$. We have to show, that $H(z)\leq m+1$. If $\|z\|\leq 2M$ the result follows by definition of $m$, thus, let us assume that $\|z\|>2M$ which also implies $\lambda_{k+1}>0$. Note that
    \[
        z=\left(\sum_{i=1}^k \lambda_i\right) \underbrace{\frac{\sum_{i=1}^k \lambda_i x_i}{\sum_{i=1}^k \lambda_i}}_{\in \Ball{2M}{0}} + \left(1-\sum_{i=1}^k \lambda_i\right) y_n.
    \]
    That is, $z$ is a convex combination of a single element in $\Ball{2\ConvexityRadius}{0}$ and $y_n$. It is easy to see that since $\|z\|>2\ConvexityRadius$ we can also find $x\in\R^d$ with $\|x\|=2\ConvexityRadius$ such that $z$ is a convex combination of this specific $x$ and $y_n$, that is, there exists $\mu\in [0,1]$ such that $z=\mu x + (1-\mu)y_n$. But then, convexity outside a ball implies $H(z)\leq m+1$, yielding that $\text{co}(\Ball{2\ConvexityRadius}{0}\cup \left\{ y_n\right\})\subset \{H\leq m+1\}$.
    
    Since the Lebesgue measure of $\text{co}(\Ball{2M}{0}\cup \left\{ y_n\right\})$ tends to infinity as $n\rightarrow\infty$, we obtain a contradiction to $\int \exp{\bigl(-H(x)\bigr)}\, \dd x<\infty$ (\cf \cite[Lemma 4]{durmus2018efficient}).
    As a consequence the set $\{H\leq m+1\}$ is bounded, \ie, there exists $R>2M$ such that for $\|x\|\geq R$ it follows $H(x)>m+1$.
    
    Now take $x \in \BallC{R}{0}$.
    We can write $R\frac{x}{\|x\|} = (1-\lambda)2M\frac{x}{\|x\|} + \lambda x$ with $\lambda=\frac{R-2M}{\|x\|-2M}$. It follows by convexity outside a ball, that
    \begin{equation}
        H\bigl(R\tfrac{x}{\|x\|}\bigr)\leq \lambda H(x) + (1-\lambda)H\bigl(2M\tfrac{x}{\|x\|}\bigr)
    \end{equation}
    and as a result
    \begin{equation}
        \begin{aligned}
            H(x) - H\bigl(2M\tfrac{x}{\|x\|}\bigr) &\geq \frac{1}{\lambda}\Bigl(H\bigl(R\tfrac{x}{\|x\|}\bigr) - H\bigl(2M\tfrac{x}{\|x\|}\bigr)\Bigr)\\
            &\geq \frac{\|x\|-2M}{R-2M}\Bigl(\inf_{\|y\|=R}H(y) - \sup_{\|y\|=2M}H(y)\Bigr).
        \end{aligned}
    \end{equation}
    Since $H(y)\leq m$ in $\Ball{2\ConvexityRadius}{0}$ and $H(y)>m+1$ for $\|x\|\geq R$ the result follows.
\end{proof}

\begin{lemma}\label{lem:prox_growth_bound}
    Let $G$ be superexponential, then for a sufficiently large $x \in \R^d$, it follows that $\|\prox_{tG}(x)\|\leq \|x\|+c$ for some $c>0$.
\end{lemma}
\begin{proof}
    Let $p\in\prox_{tG}(x)$.
    By definition of the proximal map there exists a $v\in\partial G(p)$ such that $0=\frac{1}{t}(p-x) + v$.
    Since $G$ is superexponential, there exist $x^*\in\R^d$, $\rho> 0$, and $M_\rho >0$ such that $\|x-x^*\|\geq M_\rho$ implies $\langle v, x-x^*\rangle \geq \rho\|x - x^*\|^2$, $v\in\partial G(x)$.
    Now let $R>0$ be such that \( x \in \BallC{R}{0} \) implies $\|p-x^*\|\geq M_\rho$ (\cf \Cref{lemma:prox_coercive}).
    Then, we can conclude for such \( x \) that
    \begin{equation}
        \begin{aligned}
        \|x-x^*\|^2 &= \|p-x^*\|^2 + 2t\langle p-x^*,v\rangle + t^2\|v\|^2\\
                    &\geq \|p-x^*\|^2.
        \end{aligned}
    \end{equation}
\end{proof}

\section{Implementational Details}

\begin{algorithm}
  \begin{algorithmic}[1]
    \Require Number of iterations $K$, initial condition $x^{0}$, initial $L_0$, number of backtracking iterations $J$, $\beta \in (0, 1)$, $\gamma > 1$, relative tolerance $r$
    \State $x^{-1} = x^{0}$
    \For{$k = 0, 1, \dots, K-1$}
    \State $\bar{x} = x^k + ( x^k  - x^{k-1}) / \sqrt{2}$
    \For{$j = 0, 1, \dots, J-1$} \Comment{Lipschitz backtracking procedure~\cite{cocain2020}}
        \State $x^{k+1} = \operatorname{prox}_{\frac{1}{L_k} g}( x - \nabla f(\bar{x} /L_k))$
        \If{$f(x^{k+1}) \leq  f( \bar{x}) + \langle \nabla f(\bar{x}), x^{k+1} - \bar{x} \rangle + \frac{L_k}{2}\|{\bar{ x} -  x^{k+1}}\|^2$}
            \State $L_k = \beta L_k$
            \State \textbf{break}
        \EndIf
            \State $L_{k} = \gamma L_k$
    \EndFor
    \If{$\| x^k - x^{k+1} \| / \| x^k \| \leq r$} \Comment{Stopping criterion}
        \State \textbf{break}
    \EndIf
    \EndFor
    \State \Return $x^k$
  \end{algorithmic}
  \caption{\Gls{apgd} algorithm with Lipschitz backtracking.}
  \label{algo:apgd}
\end{algorithm}
\end{document}

%% file: ex_shared.tex
\usepackage{lipsum}
\usepackage{amsfonts}
\usepackage{graphicx}
\usepackage{epstopdf}
\usepackage{csvsimple-l3}

\ifpdf
  \DeclareGraphicsExtensions{.eps,.pdf,.png,.jpg}
\else
  \DeclareGraphicsExtensions{.eps}
\fi

\usepackage{enumitem}
\setlist[enumerate]{leftmargin=.5in}
\setlist[itemize]{leftmargin=.5in}

\usepackage{cite}
\usepackage{amsmath,amssymb}
\usepackage{textcomp}
\usepackage{xcolor}
\usepackage{mathtools}
\usepackage[capitalize]{cleveref}
\usepackage{algorithm}
\usepackage[noEnd,indLines,commentColor=gray]{algpseudocodex}
\newcounter{alglinenumber}
\crefname{alglinenumber}{line}{lines}
\Crefname{alglinenumber}{Line}{Lines}

\usepackage{microtype}
\usepackage{siunitx}
\usepackage{csquotes}
\usepackage{pgfplots}
\pgfplotsset{compat=1.18}
\usepgfplotslibrary{colorbrewer}
\usepgfplotslibrary{colormaps}
\usepackage{glossaries-extra}
\newabbreviation{ula}{ULA}{unadjusted Langevin algorithm}
\newabbreviation{ald}{ALD}{annealed Langevin dynamics}
\newabbreviation{myula}{MYULA}{Moreau-Yosida regularized unadjusted Langevin algorithm}
\newabbreviation{daz}{DAZ}{diffusion at absolute zero}
\newabbreviation{apgd}{APGD}{accelerated proximal gradient descent}
\newabbreviation{dsm}{DSM}{denoising score matching}
\newabbreviation{sde}{SDE}{stochastic differential equation}
\newabbreviation{map}{MAP}{maximim a-posteriori}
\newabbreviation{mcmc}{MCMC}{Markov chain Monte Carlo}
\newabbreviation{mmse}{MMSE}{minimum mean-squared-error}
\newabbreviation{tdv}{TDV}{total deep variation}
\newabbreviation{mri}{MRI}{magnetic resonance imaging}
\newabbreviation{iid}{i.i.d.}{independent and identically distributed}
\newabbreviation{lsc}{l.sc.}{lower semicontinuous}
\newabbreviation{em}{EM}{Euler-Maruyama}
\newabbreviation{tv}{TV}{total variation}
\newabbreviation{bp}{BP}{belief propagation}
\newabbreviation{acf}{ACF}{auto-correlation function}
\newabbreviation{rock}{SK-ROCK}{stochastic
orthogonal Runge--Kutta--Chebyshev method}
\glsdisablehyper
\definecolor{siamblue}{RGB}{38, 82, 132}
\definecolor{siamgreen}{RGB}{56, 126, 33}
\definecolor{amaranth}{RGB}{228, 154, 176}
\definecolor{melon}{RGB}{236, 184, 165}
\definecolor{orangi}{RGB}{255, 193, 7}
\usetikzlibrary{arrows.meta,positioning}

\definecolor{red}{RGB}{228, 26, 28}
\definecolor{blue}{RGB}{55, 126, 184}
\definecolor{green}{RGB}{77, 175, 74}
\definecolor{orange}{RGB}{255, 127, 0}
\definecolor{purple}{RGB}{152, 78, 163}
\definecolor{brown}{RGB}{166, 86, 40}
\definecolor{pink}{RGB}{247, 129, 191}

\pgfplotscreateplotcyclelist{plotcycle}{
    {red},
    {blue},
    {green},
    {orange},
    {purple},
    {brown},
    {pink}
}

\pgfplotscreateplotcyclelist{coolwarm}{
    {siamblue},
    {amaranth},
    {melon},
    {siamgreen},
    {orangi}
}

\pgfplotsset{every axis plot/.append style={line width=.8pt}}
\tikzset{font={\fontsize{8pt}{12}\selectfont}}
\newcommand{\drawcolorbar}{
	\pgfplotscolorbardrawstandalone[
		scale=0.22, colormap={blackwhite}{gray(0cm)=(0); gray(1cm)=(1)},
		colorbar horizontal, colorbar style={ticks=none},
	]%
}

\newcommand{\R}{\mathbb{R}}
\newcommand{\N}{\mathbb{N}}
\newcommand{\Wc}{\mathcal{W}}
\newcommand{\Pc}{\mathcal{P}}

\newcommand{\Nc}{\mathcal{N}}
\newcommand{\Bc}{\mathcal{B}}
\newcommand{\E}{\mathbb{E}}

\newcommand{\Ac}{\mathcal{A}}

\newcommand{\1}{1}
\newcommand{\Fourier}{\mathcal{F}}
\newcommand{\ModulusWeakConvexity}{\rho_G}
\newcommand{\ConvexityRadius}{M}
\newcommand{\tmax}{t_{\mathrm{max}}}
\newcommand{\dd}{\mathrm{d}}
\newcommand{\Ball}[2]{B_{#1}(#2)}
\newcommand{\BallC}[2]{B^c_{#1}(#2)}

\newcommand{\intbound}{K_{\mathrm{max}}}
\newcommand{\MoreauName}{Moreau}

\DeclareMathOperator{\prox}{prox}
\DeclareMathOperator{\TV}{TV}

\DeclareMathOperator{\acf}{ACF}
\DeclareMathOperator{\NC}{NC}

\newcommand{\ie}{\textit{i.e.}}
\newcommand{\eg}{\textit{e.g.}}
\newcommand{\cf}{\textit{cf.}}

\newsiamremark{remark}{Remark}
\newsiamremark{hypothesis}{Hypothesis}
\crefname{hypothesis}{Hypothesis}{Hypotheses}
\newsiamremark{ass}{Assumption}
\crefname{ass}{Assumption}{Assumptions}

\newsiamthm{claim}{Claim}
\newlist{subassumption}{enumerate}{1}
\crefname{subassumptioni}{Assumption}{Assumptions}
\Crefname{subassumptioni}{Assumption}{Assumptions}
\setlist[subassumption,1]{labelindent=.8cm,leftmargin=*,label=\arabic*.,ref=4.3.\arabic*}

\headers{Diffusion at Absolute Zero}{A. Habring, A. Falk, M. Zach, T. Pock}

\title{Diffusion at Absolute Zero: Langevin Sampling \\Using Successive Moreau Envelopes\thanks{Submitted to the editors on March 27, 2025. The concepts presented in this manuscript have already been introduced in a (very preliminary) workshop paper submitted to \emph{2025 IEEE Statistical Signal Processing Workshop}. This version of the manuscript is available at \cite{habring2025diffusion}.
}}

\author{%
    Andreas Habring\thanks{%
        Institute of Visual Computing, Graz University of Technology (\email{andreas.habring@tugraz.at}, \email{falk@tugraz.at}, \email{thomas.pock@tugraz.at}).%
    }
    \and Alexander Falk\footnotemark[2]
    \and Martin Zach\thanks{%
        Biomedical Imaging Group, École Polytechnique Fédérale de Lausanne, 1015 Lausanne, Switzerland and Center for Biomedical Imaging, 1015 Lausanne, Switzerland (\email{martin.zach@epfl.ch})
    }
    \and Thomas Pock\footnotemark[2]%
}

\usepackage{amsopn}

\definecolor{c1}{HTML}{648FFF}
\definecolor{c2}{HTML}{DC267F}
\definecolor{c3}{HTML}{FE6100}
\definecolor{c4}{HTML}{FFB000}



%% file: tikz_code/compare_moreau.tikz
\begin{tikzpicture}[baseline] 
        \begin{axis}[width=0.5\textwidth,title={}, cycle list name=plotcycle, xlabel={Lag $k$},ymin=-0.05,ymax=0.35,ylabel = {$\|\;\cdot\;-\pi\|_{\mathrm{TV}}$},xlabel = {Moreau parameter $t$},legend pos=outer north east,legend style={draw=none},legend cell align={left}]
        \foreach \yy in {TV_err_true,0.001,0.1,0.5} {
    		\addplot table [x=t, y=\yy, col sep=comma] {images/laplace_compare_moreau.csv};
            }
            \legend{$\pi^t$, $\pi^t_\tau$ ($\tau=0.001$), $\pi^t_\tau$ ($\tau=0.1$), $\pi^t_\tau$ ($\tau=0.5$)}
    	\end{axis}
    \end{tikzpicture}

%% file: tikz_code/TV_error_gmm_rand_init.tikz
\begin{tikzpicture}[baseline] 
    	\begin{axis}[width=0.54\textwidth,title={Initial distribution $\Nc(0,1)$.}, cycle list name=plotcycle, xlabel={Iteration $k$}, ylabel={$\| \pi_k - \pi \|_{\mathrm{TV}}$},ymax=2.1,ymin=0.2, /tikz/line join=bevel,]
    		\foreach \yy in {ULA,Myula,skrock,annealing,daz,daz_skrock,gt} {
    			\addplot table [x=iter, y=\yy, col sep=comma] {images/mixture/TV_error_gmm_rand_init.csv};
    		}
    	\end{axis}
\end{tikzpicture}

%% file: tikz_code/TV_error_gmm_zero_init.tikz
\begin{tikzpicture}[baseline]
        \begin{axis}[width=0.54\textwidth,title={Initial distribution $\delta_0$.}, cycle list name=plotcycle, xlabel={Iteration $k$},ymax=2.1,ymin=0.2,yticklabel={\empty}, /tikz/line join=bevel,]
    		\foreach \yy in {ULA,Myula,skrock,annealing,daz,daz_skrock,gt} {
    			\addplot table [x=iter, y=\yy, col sep=comma] {images/mixture/TV_error_gmm_zero_init.csv};
    		}
    		\legend{ULA, MYULA, SK-ROCK, ALD, DAZ, DAZ-SK-ROCK, GT}
    	\end{axis}
\end{tikzpicture}

%% file: tikz_code/tv_prior_tv_distance_05.tex
\begin{tikzpicture}[baseline]
    \begin{axis}[
        width=0.37\textwidth,
        title={Median},cycle list name=plotcycle, /tikz/line join=bevel, xlabel={Iteration $k$}, ylabel={$\| \Pi_{(i+1) \to i}(\pi_k) - \Pi(\pi) \|_{\mathrm{TV}}$}, ymin=0,ymax=0.5]
        \addplot+ table [x=Iterations, y=0.5, col sep=comma] {images/TV_prior/ula/tv_distance.csv};
        \addplot+ table [x=Iterations, y=0.5, col sep=comma] {images/TV_prior/myula/tv_distance.csv};
        \addplot+ table [x=Iterations, y=0.5, col sep=comma] {images/TV_prior/skrock/tv_distance.csv};
        \addplot+ table [x=Iterations, y=0.5, col sep=comma] {images/TV_prior/ald/tv_distance.csv};
        \addplot+ table [x=Iterations, y=0.5, col sep=comma] {images/TV_prior/daz/tv_distance.csv};
        \addplot+ table [x=Iterations, y=0.5, col sep=comma] {images/TV_prior/daz-skrock/tv_distance.csv};
    \end{axis}
\end{tikzpicture}

%% file: tikz_code/tv_prior_tv_distance_095.tex
\begin{tikzpicture}[baseline]
    \begin{axis}[
        width=0.37\textwidth,
        title={Maximum}, xlabel={Iteration $k$},cycle list name=plotcycle, /tikz/line join=bevel, yticklabel={\empty},ymin=0,ymax=0.5]
        \addplot+ table [x=Iterations, y=0.95, col sep=comma] {images/TV_prior/ula/tv_distance.csv};
        \addplot+ table [x=Iterations, y=0.95, col sep=comma] {images/TV_prior/myula/tv_distance.csv};
        \addplot+ table [x=Iterations, y=0.95, col sep=comma] {images/TV_prior/skrock/tv_distance.csv};
        \addplot+ table [x=Iterations, y=0.95, col sep=comma] {images/TV_prior/ald/tv_distance.csv};
        \addplot+ table [x=Iterations, y=0.95, col sep=comma] {images/TV_prior/daz/tv_distance.csv};
        \addplot+ table [x=Iterations, y=0.95, col sep=comma] {images/TV_prior/daz-skrock/tv_distance.csv};
    \end{axis}
\end{tikzpicture}

%% file: tikz_code/tv_prior_tv_distance_005.tex
\begin{tikzpicture}[baseline]
    \begin{axis}[
        width=0.37\textwidth,
        title={Minimum}, xlabel={Iteration $k$},cycle list name=plotcycle, /tikz/line join=bevel, yticklabel={\empty},ymin=0,ymax=0.5]
        \addplot+ table [x=Iterations, y=0.05, col sep=comma] {images/TV_prior/ula/tv_distance.csv};
        \addplot+ table [x=Iterations, y=0.05, col sep=comma] {images/TV_prior/myula/tv_distance.csv};
        \addplot+ table [x=Iterations, y=0.05, col sep=comma] {images/TV_prior/skrock/tv_distance.csv};
        \addplot+ table [x=Iterations, y=0.05, col sep=comma] {images/TV_prior/ald/tv_distance.csv};
        \addplot+ table [x=Iterations, y=0.05, col sep=comma] {images/TV_prior/daz/tv_distance.csv};
        \addplot+ table [x=Iterations, y=0.05, col sep=comma] {images/TV_prior/daz-skrock/tv_distance.csv};
        \legend{\tiny{ULA}, \tiny{MYULA}, \tiny{SK-ROCK}, \tiny{ALD}, \tiny{DAZ}, \tiny{DAZ-SK-ROCK}}
    \end{axis}
\end{tikzpicture}

%% file: tikz_code/tv_prior_marginal_05.tex
\begin{tikzpicture}[baseline]
    \begin{axis}[
        width=0.37\textwidth,
        title={Median}, xlabel={$(Dx)_j=x_{j+1}-x_j$},cycle list name=plotcycle, /tikz/line join=bevel, ylabel={$-\log(\Pi_{(j+1)\to j}(\pi_{1000}))$}, xmin=-4, xmax=4,ymin=-0.2,ymax=8]
        \addplot+ table [x=x, y=y, col sep=comma] {images/TV_prior/ula/marginals_5.csv};
        \addplot+ table [x=x, y=y, col sep=comma] {images/TV_prior/myula/marginals_5.csv};
        \addplot+ table [x=x, y=y, col sep=comma] {images/TV_prior/skrock/marginals_5.csv};
        \addplot+ table [x=x, y=y, col sep=comma] {images/TV_prior/ald/marginals_5.csv};
        \addplot+ table [x=x, y=y, col sep=comma] {images/TV_prior/daz/marginals_5.csv};
        \addplot+ table [x=x, y=y, col sep=comma] {images/TV_prior/daz-skrock/marginals_5.csv};
        \addplot+ table [x=x, y=y, col sep=comma] {images/TV_prior/daz-skrock/marginals_gt.csv};
    \end{axis}
\end{tikzpicture}

%% file: tikz_code/tv_prior_marginal_095.tex
\begin{tikzpicture}[baseline]
    \begin{axis}[
        width=0.37\textwidth,
        title={Maximum}, xlabel={$(Dx)_j=x_{j+1}-x_j$},cycle list name=plotcycle, /tikz/line join=bevel, xmin=-4, xmax=4, yticklabel={\empty},ymin=-0.2,ymax=8]
        \addplot+ table [x=x, y=y, col sep=comma] {images/TV_prior/ula/marginals_95.csv};
        \addplot+ table [x=x, y=y, col sep=comma] {images/TV_prior/myula/marginals_95.csv};
        \addplot+ table [x=x, y=y, col sep=comma] {images/TV_prior/skrock/marginals_95.csv};
        \addplot+ table [x=x, y=y, col sep=comma] {images/TV_prior/ald/marginals_95.csv};
        \addplot+ table [x=x, y=y, col sep=comma] {images/TV_prior/daz/marginals_95.csv};
        \addplot+ table [x=x, y=y, col sep=comma] {images/TV_prior/daz-skrock/marginals_95.csv};
        \addplot+ table [x=x, y=y, col sep=comma] {images/TV_prior/daz-skrock/marginals_gt.csv};
    \end{axis}
\end{tikzpicture}

%% file: tikz_code/tv_prior_marginal_005.tex
\begin{tikzpicture}[baseline]
    \begin{axis}[
        width=0.37\textwidth,
        title={Minimum}, xlabel={$(Dx)_j=x_{j+1}-x_j$},cycle list name=plotcycle, /tikz/line join=bevel, xmin=-4, xmax=4, yticklabel={\empty},ymin=-0.2,ymax=8]
        \addplot+ table [x=x, y=y, col sep=comma] {images/TV_prior/ula/marginals_05.csv};
        \addplot+ table [x=x, y=y, col sep=comma] {images/TV_prior/myula/marginals_05.csv};
        \addplot+ table [x=x, y=y, col sep=comma] {images/TV_prior/skrock/marginals_05.csv};
        \addplot+ table [x=x, y=y, col sep=comma] {images/TV_prior/ald/marginals_05.csv};
        \addplot+ table [x=x, y=y, col sep=comma] {images/TV_prior/daz/marginals_05.csv};
        \addplot+ table [x=x, y=y, col sep=comma] {images/TV_prior/daz-skrock/marginals_05.csv};
        \addplot+ table [x=x, y=y, col sep=comma] {images/TV_prior/daz-skrock/marginals_gt.csv};
    \end{axis}
\end{tikzpicture}

%% file: tikz_code/TV_error_TV_chain_05.tikz
\begin{tikzpicture}[baseline] 
        \begin{axis}[width=0.37\textwidth,title={50th percentile}, cycle list name=plotcycle, xlabel={Iteration \( k \)},ymin=0,ymax=2.2,ylabel={$\| \Pi_{i} (\pi_k) - \Pi_{i} (\pi)\|_{\mathrm{TV}}$}]
    		\foreach \yy in {ULA,Myula,skrock,annealing,daz,daz_skrock} {
    			\addplot table [x=iter, y=\yy, col sep=comma] {images/TV_chain/TV_error_q_05_TV_chain.csv};
    		}
    	\end{axis}
    \end{tikzpicture}

%% file: tikz_code/TV_error_TV_chain_095.tikz
\begin{tikzpicture}[baseline] 
        \begin{axis}[width=0.37\textwidth,title={95th percentile}, cycle list name=plotcycle, xlabel={Iteration \( k \)}, ymin=0,ymax=2.2,yticklabel={\empty}]
    		\foreach \yy in {ULA,Myula,skrock,annealing,daz,daz_skrock} {
    			\addplot table [x=iter, y=\yy, col sep=comma] {images/TV_chain/TV_error_q_095_TV_chain.csv};
    		}
    	\end{axis}
    \end{tikzpicture}

%% file: tikz_code/TV_error_TV_chain_005.tikz
\begin{tikzpicture}[baseline] 
        \begin{axis}[width=0.37\textwidth,title={5th percentile}, cycle list name=plotcycle, xlabel={Iteration \( k \)},ymin=0,ymax=2.2,yticklabel={\empty}]
    		\foreach \yy in {ULA,Myula,skrock,annealing,daz,daz_skrock} {
    			\addplot table [x=iter, y=\yy, col sep=comma] {images/TV_chain/TV_error_q_005_TV_chain.csv};
    		}
            \legend{\tiny{ULA}, \tiny{MYULA}, \tiny{SK-ROCK}, \tiny{ALD}, \tiny{DAZ}, \tiny{DAZ-SK-ROCK}}
    	\end{axis}
    \end{tikzpicture}

%% file: tikz_code/TV_error_TV_image_05.tikz
\begin{tikzpicture}[baseline] 
        \begin{axis}[width=0.37\textwidth,title={50th percentile}, cycle list name=plotcycle, xlabel={Iteration \( k \)},ymin=0,ymax=2.2,ylabel={$\| \Pi_{i} (\pi_k) - \Pi_{i} (\pi)\|_{\mathrm{TV}}$}]
    		\foreach \yy in {ULA,Myula,skrock,annealing,daz,daz_skrock} {
    			\addplot table [x=iter, y=\yy, col sep=comma] {images/TV_image/TV_error_q_05_TV_image.csv};
    		}
    	\end{axis}
    \end{tikzpicture}

%% file: tikz_code/TV_error_TV_image_095.tikz
\begin{tikzpicture}[baseline] 
        \begin{axis}[width=0.37\textwidth,title={95th percentile}, cycle list name=plotcycle, xlabel={Iteration \( k \)}, ymin=0,ymax=2.2,yticklabel={\empty}]
    		\foreach \yy in {ULA,Myula,skrock,annealing,daz,daz_skrock} {
    			\addplot table [x=iter, y=\yy, col sep=comma] {images/TV_image/TV_error_q_095_TV_image.csv};
    		}
    	\end{axis}
    \end{tikzpicture}

%% file: tikz_code/TV_error_TV_image_005.tikz
\begin{tikzpicture}[baseline] 
        \begin{axis}[width=0.37\textwidth,title={5th percentile}, cycle list name=plotcycle, xlabel={Iteration \( k \)},ymin=0,ymax=2.2,yticklabel={\empty}]
    		\foreach \yy in {ULA,Myula,skrock,annealing,daz,daz_skrock} {
    			\addplot table [x=iter, y=\yy, col sep=comma] {images/TV_image/TV_error_q_005_TV_image.csv};
    		}
            \legend{\tiny{ULA}, \tiny{MYULA}, \tiny{SK-ROCK}, \tiny{ALD}, \tiny{DAZ}, \tiny{DAZ-SK-ROCK}}
    	\end{axis}
    \end{tikzpicture}

%% file: tikz_code/typical_set_TV_chain.tikz
\begin{tikzpicture}[baseline] 
        \begin{axis}[width=0.5\textwidth,title={Chain}, cycle list name=plotcycle, xlabel={Iteration \( k \)},ylabel={$U(x)$}]
    		\foreach \yy in {ULA,Myula,skrock,annealing,daz,daz_skrock} {
    			\addplot table [x=iter, y=\yy, col sep=comma] {images/TV_chain/pot_vals_TV_chain.csv};
    		}
    	\end{axis}
    \end{tikzpicture}

%% file: tikz_code/typical_set_TV_image.tikz
\begin{tikzpicture}[baseline] 
        \begin{axis}[width=0.5\textwidth,title={Image}, cycle list name=plotcycle, xlabel={Iteration \( k \)},ylabel={}]
    		\foreach \yy in {ULA,Myula,skrock,annealing,daz,daz_skrock} {
    			\addplot table [x=iter, y=\yy, col sep=comma] {images/TV_image/pot_vals_TV_image.csv};
    		}
    		\legend{ULA,Myula,SK-ROCK,ALD,DAZ,DAZ-SK-ROCK}
    	\end{axis}
    \end{tikzpicture}